\documentclass[11pt,english,reqno]{amsart}

\usepackage{layout}
\usepackage{amsmath,amscd,amsthm,amssymb,amsfonts,bm}
\usepackage{newlfont}
\usepackage{mathrsfs}
\usepackage[dvips]{graphicx}
\usepackage{epsf}

\usepackage{color}
\usepackage{hhline}
\usepackage{rotating}
\usepackage{verbatim}
\usepackage{lipsum}

\usepackage[bookmarks=false,colorlinks=true,linkcolor=blue,urlcolor=blue,citecolor=blue,filecolor=blue,pdfstartview=FitH]{hyperref}

\RequirePackage[hyperpageref]{backref}
\renewcommand*{\backref}[1]{}
\renewcommand*{\backrefalt}[4]{
     \ifcase #1 (no cited)
      \or (cited p. #2)
      \else (cited pp. #2)
      \fi}
      
\usepackage{enumitem}
\usepackage{nccmath}
\usepackage[normalem]{ulem}

\usepackage[footnotesize]{subfigure}
\usepackage{caption}
\newcommand{\E}{\mathbb E}

\newcommand{\bH}{\mathbb H}

\renewcommand{\P}{\mathbb P}\newcommand{\R}{\mathbb R}

\newcommand{\indic}{1\negthickspace\text{I}}\newcommand{\indicat}{{\text{1}}\negthickspace\text{I}}

\newcommand{\cB}{\cal{B}}\newcommand{\cC}{\cal{C}}

\newcommand{\cK}{\cal{K}}\newcommand{\cN}{\cal{N}}
\newcommand{\cP}{\cal{P}}
\newcommand{\cV}{\cal{V}}

\newcommand{\eps}{\varepsilon}\newcommand{\fhi}{\varphi}\newcommand{\la}{\lambda}
\newcommand{\dd}{\mathrm{d}}

\newcommand{\pass}{\vspace{0.3cm}\noindent}\newcommand{\noi}{\noindent}

\DeclareMathOperator{\conv}{conv}

\theoremstyle{plain}
\newtheorem{theo}{Theorem}[section]
\newtheorem{prop}[theo]{Proposition}
\newtheorem{lem}[theo]{Lemma}
\newtheorem{cor}[theo]{Corollary}

\theoremstyle{definition}

\definecolor{note}{rgb}{1,0,0}

\author[P. Calka, Y. Demichel and N. Enriquez]{Pierre Calka\textsuperscript{1}, Yann Demichel\textsuperscript{2} and Nathana\"el Enriquez\textsuperscript{3}}

\address{\textsuperscript{1} Laboratoire de Math\'ematiques Rapha\"el Salem, UMR CNRS 6085, Universit\'e de Rouen Normandie, avenue de l'Universit\'e, Technop\^ole du Madrillet, 76801 Saint-Etienne-du-Rouvray, France.}
\email{pierre.calka@univ-rouen.fr}

\address{\textsuperscript{2} Laboratoire MODAL'X, UMR CNRS 9023, UPL, Universit\'e Paris Nanterre, 200 avenue de la R\'e\-pu\-bli\-que, 92001 Nanterre, France.}
\email{yann.demichel@parisnanterre.fr}

\address{\textsuperscript{3} Laboratoire de Math\'ematiques d'Orsay, Universit\'e Paris-Saclay, Bâtiment 307, 91405 Orsay, France.}
\email{nathanael.enriquez@universite-paris-saclay.fr}

\address{This work was partially supported by the French ANR
grant ASPAG (ANR-17-CE40-0017), the Institut Universitaire de France  and the French research group GeoSto (CNRS-RT3477).}

\title[Elongated Poisson-Voronoi cells in an empty half-plane]{Elongated Poisson-Voronoi cells in an empty half-plane}

\begin{document}

\begin{abstract}
The Voronoi tessellation of a homogeneous Poisson point process in the lower half-plane of $\R^2$ gives rise to a family of vertical elongated cells in the upper half-plane. The set of edges of these cells is ruled by a Markovian branching mechanism which is asymptotically described by two sequences of iid variables which are respectively Beta and exponentially distributed. This leads to a precise description of the scaling limit of a so-called \textit{typical} cell. The limit object is a random apeirogon that we name \textit{menhir} in reference to the Gallic huge stones. We also deduce from the aforementioned branching mechanism that the number of vertices of a cell of height $\la$ is asymptotically equal to $\frac45\log\la$.
\end{abstract}


\subjclass[2020]{Primary 60D05, 52A22; Secondary  60G55, 60J05, 33B15, 39A50, 52A23}
\keywords{Poisson point process; Poisson-Voronoi tessellation; Markov chain; Markov coupling; Random difference equations}

\maketitle

\fussy

\addtocontents{toc}{\vspace{0.25cm}}%
\section{Introduction}\label{sec:intro}

\noi The Voronoi tessellation generated by a locally finite set of points in $\R^2$ called nuclei is the partition of the plane by convex polygons called Voronoi cells defined as the sets of locations closer to a particular nucleus than to the others. 
Large cells in Voronoi tessellations generated by a homogeneous Poisson point process have attracted a lot of attention for decades. The paper \cite{hrs04} proves and extends D.~G.~Kendall's statement which asserts that large cells from a stationary and isotropic Poisson line tessellation are close to the circular shape. In \cite{hs07}, this fact is proved to occur for Poisson-Voronoi tessellations as well. Thereafter, the work \cite{cs05} investigates the mean defect area and mean number of vertices of the typical Poisson-Voronoi cell conditioned on containing a disk of radius $r$ when $r\longrightarrow\infty$. 

\pass Large Voronoi cells appear in locations where there is only one nucleus inside a large domain. In \cite{CDE21}, we consider the Poisson-Voronoi tessellation generated by the union of a Poisson point process outside a large deterministic set with an isolated point belonging to that set. We then estimate the mean and variance of the area, perimeter and number of vertices of the Poisson-Voronoi cell associated with the isolated nucleus.

\begin{figure}[ht!]
\begin{center}
\includegraphics[trim=3.5cm 8cm 6cm 6cm,clip,scale=.5]{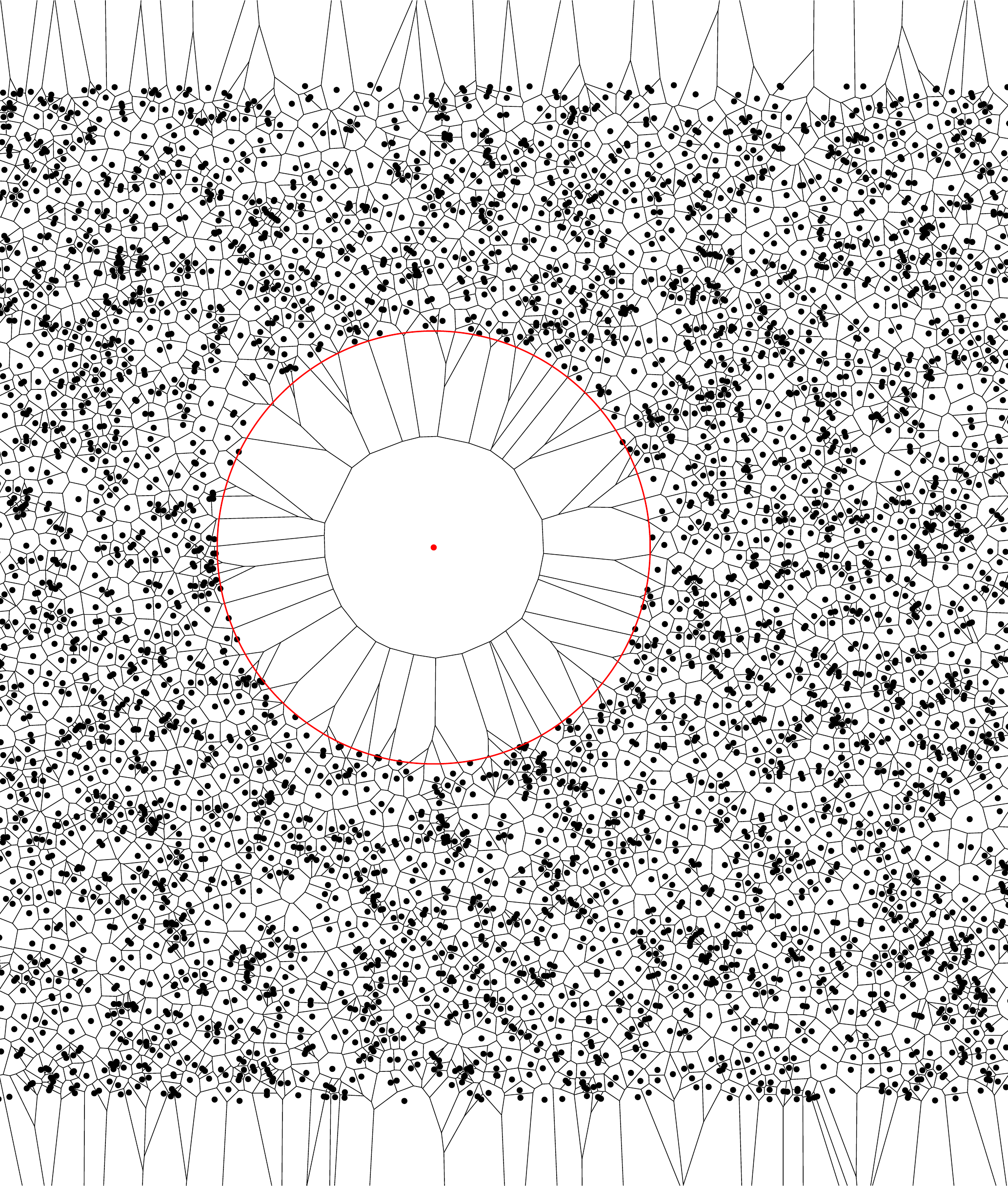}
\hspace*{0.7cm} \includegraphics[trim=4.5cm 1cm 4.5cm 0cm, clip, height=5.35cm,scale=1]{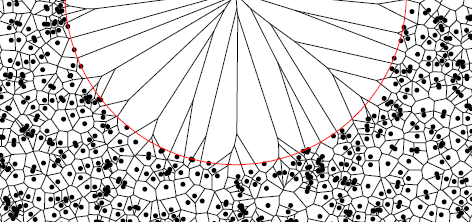}
\end{center}
\caption{Left: The Voronoi cell around an isolated nucleus in a large disk. Right: The elongated cells above a fraction of the boundary of the disk.}
\label{fig:AHL}
\end{figure}

\pass We observe that this large cell is surrounded by a branching collection of elongated cells, see Figure \ref{fig:AHL} which represents the Voronoi cell of an isolated nucleus and its neighbors in an empty large disk. These elongated cells constitute the object of our attention in the present paper. When zooming close to the boundary of the empty set, we observe that the local geometry of the elongated cells is captured by the Voronoi tessellation generated by a Poisson point process limited to the lower half-plane of $\R^2$. In this idealized model, we focus our study on a so-called typical cell whose highest vertex is at height $\la$ and show that its boundary all the way from the highest vertex to the $x$-axis is ruled by a two-dimensional Markov chain. This is slightly reminiscent of \cite{BTZ} where walking along the consecutive nuclei of the Voronoi cells which intersect the horizontal $x$-axis also gives birth to a Markov chain. When $\la\longrightarrow\infty$, we prove that this Markov chain is asymptotically governed by two sequences of iid variables which are respectively Beta(2,2) and exponentially distributed. This leads to a precise description of the scaling limit of the typical cell. The limiting object is a random apeirogon with a unique accumulation point at its bottom, where we recall that the word \textit{apeirogon} denotes the closed convex hull of a countably infinite set of extreme points in the plane. We have chosen to name this limiting object \textit{menhir} in reference to the large man-made upright stones from the European Bronze Age, see Figure \ref{fig:Menhir}. 

\pass Let us endow the Euclidean plane with the orthogonal frame $(o,e_1,e_2)$. Let us consider a homogeneous Poisson point process $\cP$ of intensity $1$ on the lower half-plane $\bH^c$ where $\bH$ stands for $\bH=\{(x_1,x_2)\in\R^2 : x_2> 0 \}$.

\begin{figure}[ht!]
\begin{center}
\includegraphics[trim= 5.9cm 17.5cm 5.3cm 5.5cm, clip, scale=1.2,angle=180]{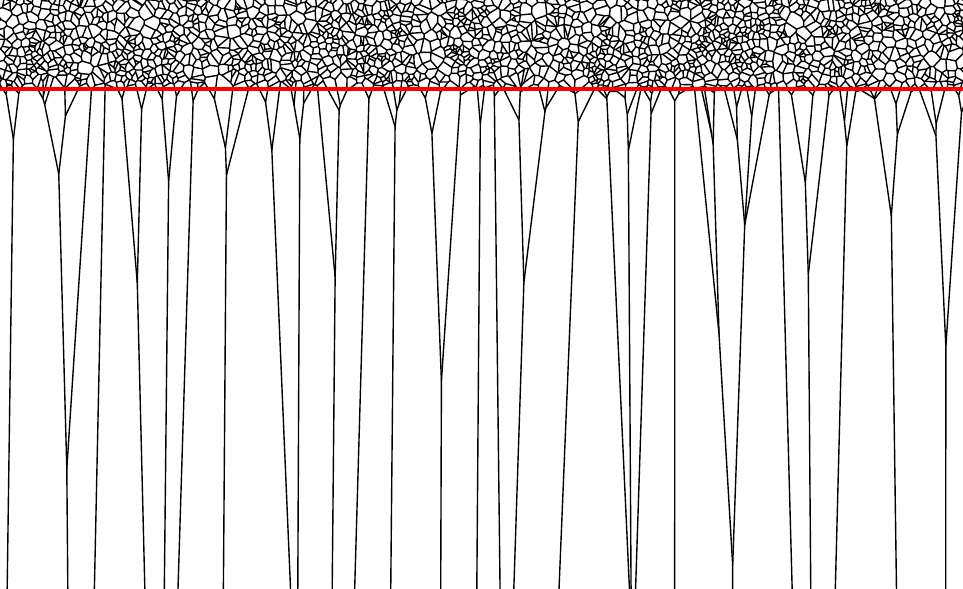}
\end{center}
\caption{A homogeneous Poisson point process $\cP$ of intensity $1$ on the lower half-plane $\bH^c$ inducing elongated Poisson-Voronoi cells on the upper half-plane $\bH$.}
\label{fig:Overwater}
\end{figure}

\noi Let $\cC^{(\la)}$ be the typical Voronoi cell with its highest vertex at $(0,\la)$ and $Z_{\textsl c}^{(\la)}$ its nucleus, see Section \ref{sec:condpois}. Let us endow the set of non-empty compact convex sets of $\R^2$ with the topology defined by the Hausdorff distance. For every $\la>0$, we consider the affine transformation $F^{(\la)}:\R^2\longrightarrow\R^2$ defined by
\begin{equation*}
F^{(\la)}(x,y)=(\la^{-\frac13}x,\la^{-1}y).
\end{equation*}

\pass The set $F^{(\la)}(\cC^{(\la)}-Z_{\textsl c}^{(\la)})$ is the renormalized typical cell seen from its nucleus $Z_{\textsl c}^{(\la)}$. Our first result is the description of its limit shape for the convergence in distribution.

\pass\textbf{Description of the limiting menhir $\cC^{(\infty)}$.}

\noi The menhir $\cC^{(\infty)}$ is the closed convex hull of the union of sets of \textit{left} and \textit{right} vertices $\{V_n^{\mathbf l}\}_n\cup\{V_n^{\mathbf r}\}_n$ which are defined iteratively below through two Markov chains $\{(B_n^{\mathbf l},H_n^{\mathbf l})\}_n$ and $\{(B_n^{\mathbf r},H_n^{\mathbf r})\}_n$, as seen in Figure \ref{fig:TargetPoints}. 

\pass\textbf{1. Initialization of two Markov chains.}
\begin{itemize}[parsep=-0.15cm,itemsep=0.25cm,topsep=0.2cm,wide=0.15cm,leftmargin=0.5cm]
\item[$\bullet$] Let $G\sim\mathrm{Gamma}(\tfrac83,1)$ be a Gamma distributed random variable.
\item[$\bullet$] Let $(U_{(1)},\ldots,U_{(6)})$ be the order statistics of a random vector $(U_1,\ldots,U_6)$ uniformly distributed in $(-1,1)^6$ and independent of $G$.
\item[$\bullet$] Let $\eps$ be a random variable independent of $G$ and $(U_1,\ldots,U_6)$ and equal to $3$ or $4$ with probability $\frac12$.
\item[$\bullet$] Let $(\Theta_{\textsl l},\Theta_{\textsl c},\Theta_{\textsl r})=\left(\tfrac{3}{2}G\right)^{\frac13}(U_{(1)},U_{(\eps)},U_{(6)})$.
\end{itemize}
The two Markov chains start at $(B_0^{\mathbf l},H_0^{\mathbf l}) =(\tfrac12 (\Theta_{\textsl c}-\Theta_{\textsl l}),1) $ and $(B_0^{\mathbf r},H_0^{\mathbf r})=(\tfrac12 (\Theta_{\textsl r}-\Theta_{\textsl c}),1)$ respectively.

\pass\textbf{2. Recursive construction of the two Markov chains.}\null

\noi We define the sequences $\{(B_n^{\mathbf l},H_n^{\mathbf l})\}_n$ and $\{(B_n^{\mathbf r},H_n^{\mathbf r})\}_n$ with the recursion identities\label{descBH}
\begin{equation}\label{eq:BnHnlrec}
(B_{n+1}^{\mathbf l},H_{n+1}^{\mathbf l})=\Big(\beta_n B_n^{\mathbf l},\mfrac{(B_n^{\mathbf l})^3H_n^{\mathbf l}}{(B_n^{\mathbf l})^3+\tfrac32\xi_n H_n^{\mathbf l}}\Big)
\text{ and }
(B_{n+1}^{\mathbf r},H_{n+1}^{\mathbf r})=\Big(\beta_n' B_n^{\mathbf r},\mfrac{(B_n^{\mathbf r})^3H_n^{\mathbf r}}{(B_n^{\mathbf r})^3+\tfrac32\xi_n' H_n^{\mathbf r}}\Big)
\end{equation}
where $\{\beta_n\}_n$, $\{\beta_n'\}_n$, $\{\xi_n\}_n$, $\{\xi_n'\}_n$ are four independent sequences of iid random variables such that $\beta_0\overset{(d)}{=}\beta_0'\sim\mbox{Beta}(2,2)$ and $\xi_0\overset{(d)}{=}\xi_0'\sim\mbox{Exp}(1)$.

\pass\textbf{3. Construction of the vertices of $\cC^{(\infty)}$ with the two Markov chains.} \null

\noi The highest vertex of the menhir is $V_0=-\Theta_{\textsl c}e_1+e_2$. We then introduce two sequences of points $\{V_n^{\mathbf l}=X_n^{\mathbf l}e_1+Y_n^{\mathbf l}e_2\}_n$ and $\{V_n^{\mathbf r}=X_n^{\mathbf r}e_1+Y_n^{\mathbf r}e_2\}_n$ by the following procedure: $V_0^{\mathbf l}=V_0^{\mathbf r}=V_0$ and for all $n\ge 1$, 
$V_n^{\mathbf l}$ (resp. $V_n^{\mathbf r}$) is the unique point with second coordinate  $Y_n^{\mathbf l}=H_n^{\mathbf l}$ (resp. $Y_n^{\mathbf r}=H_n^{\mathbf r}$)
and such that $V_n^{\mathbf l}$ (resp. $V_n^{\mathbf r}$) is on the half-line starting from $V_{n-1}^{\mathbf l}$ (resp. $V_{n-1}^{\mathbf r}$) and going to the {\it target point} at $-B_{n-1}^{\mathbf l} e_1$ (resp. $B_{n-1}^{\mathbf r} e_1)$, see Figure \ref{fig:TargetPoints}. We then define the two \textit{branches} $\cB_{\mathbf l}$ \label{defbranchBl} and $\cB_{\mathbf r}$ as the closed convex chains generated by the two sequences $\{V_n^{\mathbf l}\}_n$ and $\{V_n^{\mathbf r}\}_n$ respectively. 

\begin{figure}[ht!]
\centering
\includegraphics[scale=0.9]{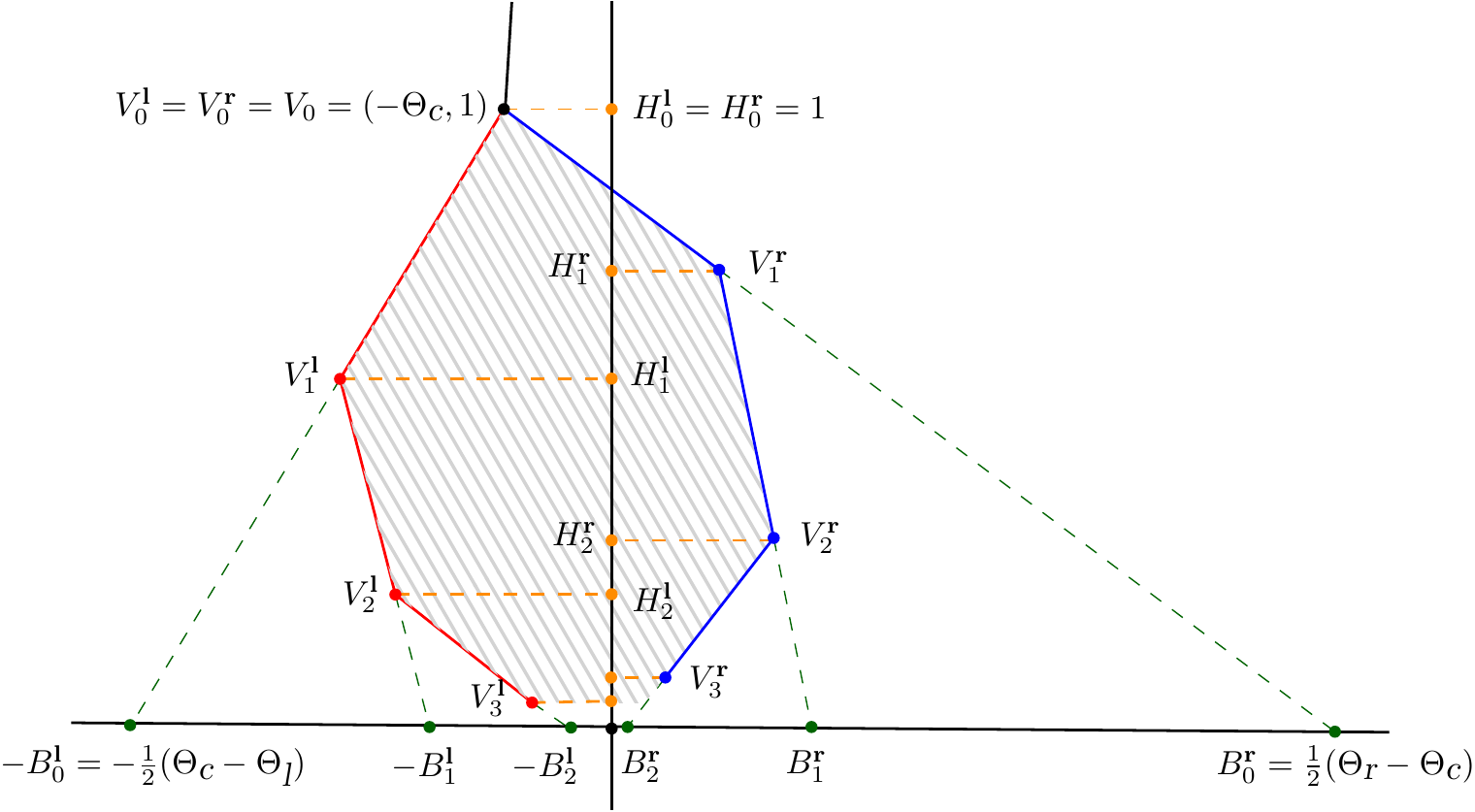}
\caption{Iterative construction of the two branches $\cB_{\mathbf l}$ and $\cB_{\mathbf r}$ of the menhir using the two sequences of target points $-B_n^{\mathbf l} e_1$ and $B_n^{\mathbf r} e_1$ respectively.}
\label{fig:TargetPoints}
\end{figure}

\pass In particular, the initial vertex is $V_0^{\mathbf l}=(X_0^{\mathbf l},Y_0^{\mathbf l})=(-\Theta_{\textsl c},1)$ and $\{X_n^{\mathbf l}\}_n$ satisfies the recursion relation\label{old2}
$X_n^{\mathbf l}-\frac{X_{n+1}^{\mathbf l}-X_n^{\mathbf l}}{H_{n+1}^{\mathbf l}-H_n^{\mathbf l}}H_n^{\mathbf l}=-B_n^{\mathbf l}$    
which leads to the formula
\begin{equation}\label{eq:defXnl}
V_n^{\mathbf l}=(X_n^{\mathbf l},Y_n^{\mathbf l})=\bigg(\negthickspace -\mfrac12(\Theta_{\textsl l}-\Theta_{\textsl c})H_n^{\mathbf l} - B_n^{\mathbf l} - H_n^{\mathbf l}\sum_{k=0}^{n-1}\mfrac{B_k^{\mathbf l}-B_{k+1}^{\mathbf l}}{H_{k+1}^{\mathbf l}},H_n^{\mathbf l}\bigg).    
\end{equation}

\noi Similarly, $V_0^{\mathbf r}=(X_0^{\mathbf r},Y_0^{\mathbf r})=(-\Theta_{\textsl c},1)$ and
\begin{equation*}
V_n^{\mathbf r}=(X_n^{\mathbf r},Y_n^{\mathbf r})=\bigg(\negthickspace -\mfrac12(\Theta_{\textsl r}+\Theta_{\textsl c})H_n^{\mathbf r}
+ B_n^{\mathbf r} + H_n^{\mathbf r}\sum_{k=0}^{n-1}\mfrac{B_k^{\mathbf r}-B_{k+1}^{\mathbf r}}{H_{k+1}^{\mathbf r}},H_n^{\mathbf r}\bigg).    
\end{equation*}

\noi The above description of the limiting menhir provides an explicit algorithm for simulating $\cC^{(\infty)}$ as seen in Figure \ref{fig:Menhir}.

\begin{figure}[ht!]
\begin{center}
\includegraphics[scale=0.545]{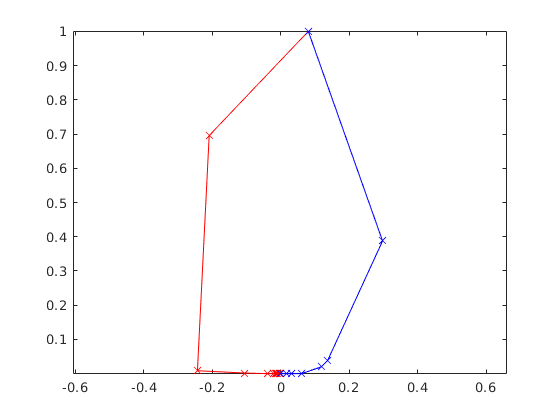}
\hspace{-0.8cm}\includegraphics[scale=0.545]{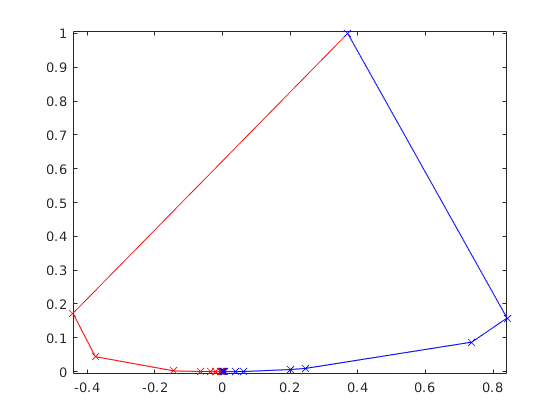}
\end{center}
\caption{Two simulations of the menhir: the red (resp. blue) polygonal line corresponds to the branch with vertices $\{V_n^{\mathbf l}\}_n$ (resp. $\{V_n^{\mathbf r}\}_n$).}
\label{fig:Menhir}
\end{figure}

\newpage

\pass We are now ready to state our main result.

\begin{theo}\label{theo:menhir} \noi

\noi $(i)$ {\it\bf Scaling limit.} The set $F^{(\la)}(\cC^{(\la)}-Z_{\textsl c}^{(\la)})$ converges in distribution for the Hausdorff metric when $\la\longrightarrow\infty$ to the apeirogon $\cC^{(\infty)}$ with a unique accumulation point at the origin.

\pass $(ii)$ {\it\bf Local behavior at the origin of the limiting menhir.} The two sequences $\Big\{\mfrac{Y_n^{\mathbf l}}{(X_n^{\mathbf l})^3}\Big\}_n$ and $\Big\{\mfrac{Y_n^{\mathbf r}}{(X_n^{\mathbf r})^3}\Big\}_n$ converge in distribution to the same non-de\-ge\-ne\-rate distribution.

\pass $(iii)$ {\it\bf Asymptotic number of vertices.} The total number of vertices $\cN(\cC^{(\la)})$ of $\cC^{(\la)}$ satisfies the following convergence when $\la\longrightarrow\infty$:
\begin{equation*}\label{eq:limitnumbersides}
\frac{\cN(\cC^{(\la)})}{\frac45\log\la}\overset{\P,L^1}{\longrightarrow} 1.
\end{equation*}
\end{theo}

\noi Our work is specific to the planar Poisson-Voronoi tessellation. As for many classical random spatial models (Wulff crystal, random planar maps, integer polytopes...), extending Theorem \ref{theo:menhir} beyond dimension two seems delicate. Our methods deeply rely on a coding via a branching process and a linear description of the boundary of $\cC^{(\la)}$ which are specific to dimension $2$. Nevertheless we expect to be able to extend to any dimension the description of the top of the cell $\cC^{(\la)}$ and notably Proposition \ref{prop:plambda}, i.e. the description of the Poisson point process conditioned on having a Voronoi vertex at height $\la$.
Another interesting line of investigation would consist in studying the whole set of cells with height larger than $\la$. For instance, we can ask about the pair-correlation function related to the point process defined as the section of the Voronoi skeleton and a horizontal line at height $\la$. This would provide additional information on the global properties of the population of high cells, as seen in Figure \ref{fig:Overwater}.

\pass{\bf Outline of the paper.} In Section \ref{sec:palm}, we make explicit the distribution of the Poisson point process conditional on having a Voronoi vertex at $(0,\la)$. We then introduce a sequence of triangles which describes the boundary of the cell $\cC^{(\la)}$ and which gives birth to a two-dimensional Markov chain. Our strategy of proof is then detailed at the end of Section \ref{sec:palm}. Section \ref{sec:coupling} is the technically most challenging part where we couple this Markov chain with an idealized Markov chain whose transition kernel can be nicely expressed in terms of two independent Beta and exponential variables. Finally, we use that coupling to prove Theorem \ref{theo:menhir} in Section \ref{sec:mainproof}.

\addtocontents{toc}{\vspace{0.25cm}}%
\section{A Markovian sequence of triangles}\label{sec:palm}

\noi In order to prove Theorem \ref{theo:menhir} we focus our study in Sections \ref{sec:palm}, \ref{sec:coupling} and \ref{sec:mainproof} on the sequence of consecutive edges of a Voronoi cell when going down from a vertex at $(0,\la)$ through the leftmost edge emanating from $(0,\la)$ and choosing afterwards at each next intersection the edge to the right until reaching the $x$-axis. This branch denoted by $\cB^{(\la)}$ is the left part of the boundary of the Voronoi cell having its highest vertex at $(0,\la)$ and ending at the first vertex below the $x$-axis. Precisely, the branch $\cB^{(\la)}$ is the polygonal line joining the first consecutive vertices $V_0^{(\la)}=(0,\la), V_1^{(\la)},\ldots, V_{\cN(\cB^{(\la)})}^{(\la)}$ where $\cN(\cB^{(\la)})$ is the number of edges of $\cB^{(\la)}$. 

\pass The vertex $V_0^{(\la)}=(0,\la)$ belongs to three Voronoi cells whose nuclei are denoted from left to right by $Z_{\textsl l}^{(\la)}$, $Z_{\textsl c}^{(\la)}$ and $Z_{\textsl r}^{(\la)}$. We are interested in the cell associated with $Z_{\textsl c}^{(\la)}$ whose highest vertex is $(0,\la)$. Notice that $\cB^{(\la)}$ starts with a portion of the bisecting line between $Z_{\textsl l}^{(\la)}$ and $Z_{\textsl c}^{(\la)}$ until it reaches the vertex $V_1^{(\la)}$ and is followed by consecutive portions $[V_n^{(\la)},V_{n+1}^{(\la)}]$ of bisecting lines between $Z_{\textsl c}^{(\la)}$ and a sequence of consecutive nuclei denoted by $\{Z_n^{(\la)}\}_n$ and such that $Z_0^{(\la)}=Z_{\textsl l}^{(\la)}$. To each edge belonging to $\cB^{(\la)}$, we associate the isosceles triangle with vertices $V_n^{(\la)}$, $Z_n^{(\la)}$ and $Z_{\textsl c}^{(\la)}$, see Figure \ref{fig:FirstTriangles} below. 

\begin{figure}[ht!]
\begin{center}
\includegraphics[scale=1]{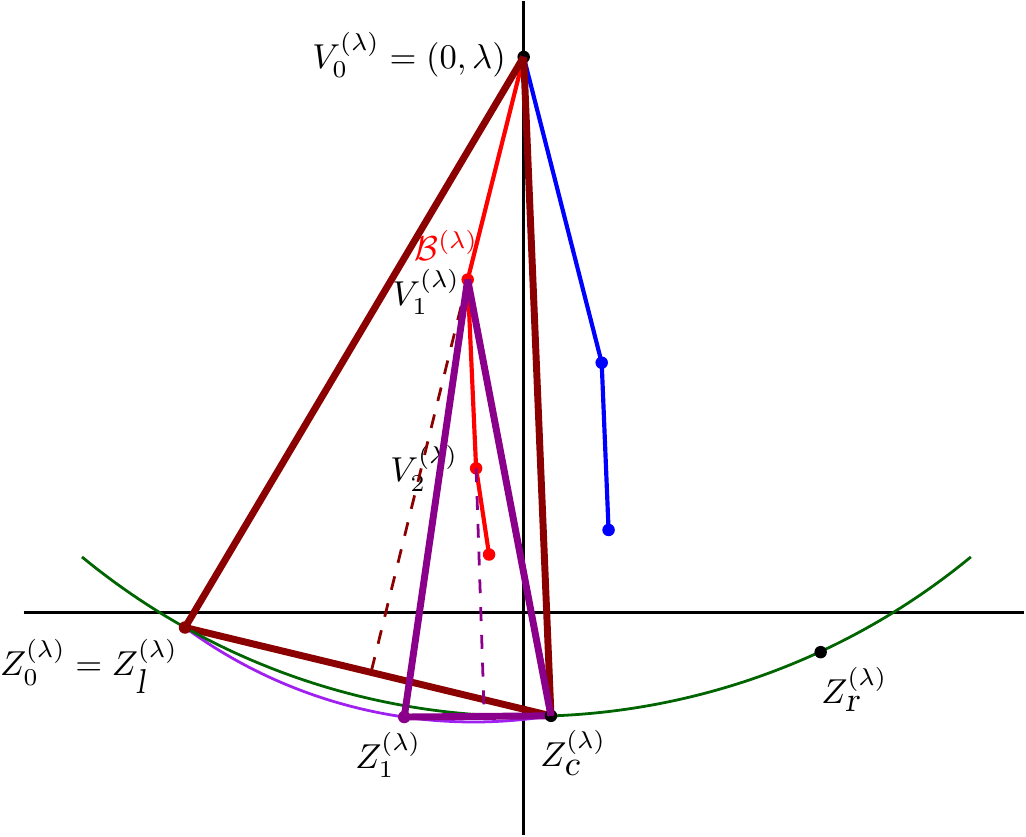}
\end{center}
\caption{The first two isosceles triangles with vertices $V_0^{(\la)}$, $Z_{\textsl l}^{(\la)}$ and $Z_c^{(\la)}$ and $V_1^{(\la)}$, $Z_1^{(\la)}$ and $Z_c^{(\la)}$ respectively.}
\label{fig:FirstTriangles}
\end{figure}

\noi In this section, we study the sequence of couples constituted with the normalized half-length of the basis and height of each of these triangles. In Sections \ref{sec:condpois} and \ref{sec:convpos}, we explain how to initiate this sequence. In Section \ref{sec:strategyproof}, we then construct the whole sequence and finally, in Section \ref{sec:bullet}, we explain how the strategy of proof of our main results will depend on its Markov properties. 

\subsection{The conditional Poisson point process when $(0,\la)$ is a vertex}\label{sec:condpois}

\pass

\noi We recall that the distribution of $\cP$ is invariant under the action of the horizontal translations. Let us introduce the point process $\cP^{(\la)}$ as the Palm version of $\cP$ conditioned on having a Voronoi vertex at the position $V_0^{(\la)}=(0,\la)$. In this context, we refer to $V_0^{(\la)}$ as a so-called \textit{typical vertex at height} $\la$. The associated nuclei $Z_{\textsl l}^{(\la)}$, $Z_{\textsl c}^{(\la)}$ and $Z_{\textsl r}^{(\la)}$ defined above are located on a circle centered at $V_0^{(\la)}$ and with radius larger than $\la$ equal to $(\la+\la^{-\frac13}R^{(\la)})$ for some positive random variable $R^{(\la)}$. Moreover, the angle between the half-lines $V_0^{(\la)}-\R_+ e_2$ and $V_0^{(\la)}+\R_+(Z_{\textsl l}^{(\la)}-V_0^{(\la)})$ (resp. $V_0^{(\la)}+\R_+(Z_{\textsl c}^{(\la)}-V_0^{(\la)})$, $V_0^{(\la)}+\R_+(Z_{\textsl r}^{(\la)}-V_0^{(\la)})$ $)$ is denoted by $\la^{-\frac23}\Theta_{\textsl l}^{(\la)}$ (resp. $\la^{-\frac23}\Theta_{\textsl c}^{(\la)}$, $\la^{-\frac23}\Theta_{\textsl r}^{(\la)}$) with $\Theta_{\textsl l}^{(\la)}<\Theta_{\textsl c}^{(\la)}<\Theta_{\textsl r}^{(\la)}$, see Figure \ref{fig:Quadruplet}. The purpose of Section \ref{sec:condpois} is to derive the distribution of $\cP^{(\la)}$ and of the quadruplet $(R^{(\la)},\Theta_{\textsl l}^{(\la)},\Theta_{\textsl c}^{(\la)},\Theta_{\textsl r}^{(\la)})$.

\begin{prop}\label{prop:plambda}
The Palm version $\cP^{(\la)}$ of $\cP$ conditioned on having a Voronoi vertex at $V_0^{(\la)}$ is
\begin{equation}\label{eq:plambda}
\cP^{(\la)}\overset{(d)}{=}\big[\cP\setminus D\big({V_0^{(\la)}}, \la+\la^{-\frac13}R^{(\la)}\big)\big]\cup \big\{Z_{\textsl l}^{(\la)},Z_{\textsl c}^{(\la)},Z_{\textsl r}^{(\la)}\big\}
\end{equation}
where the random quadruplet $(R^{(\la)},\Theta_{\textsl l}^{(\la)},\Theta_{\textsl c}^{(\la)},\Theta_{\textsl r}^{(\la)})$, which defines $(Z_{\textsl l}^{(\la)},Z_{\textsl c}^{(\la)},Z_{\textsl r}^{(\la)})$ as in Figure \ref{fig:Quadruplet}, has a
probability density function given at \eqref{eq:densityquadruplet}.
\end{prop}

\begin{figure}[ht!]
\begin{center}
\includegraphics[scale=1]{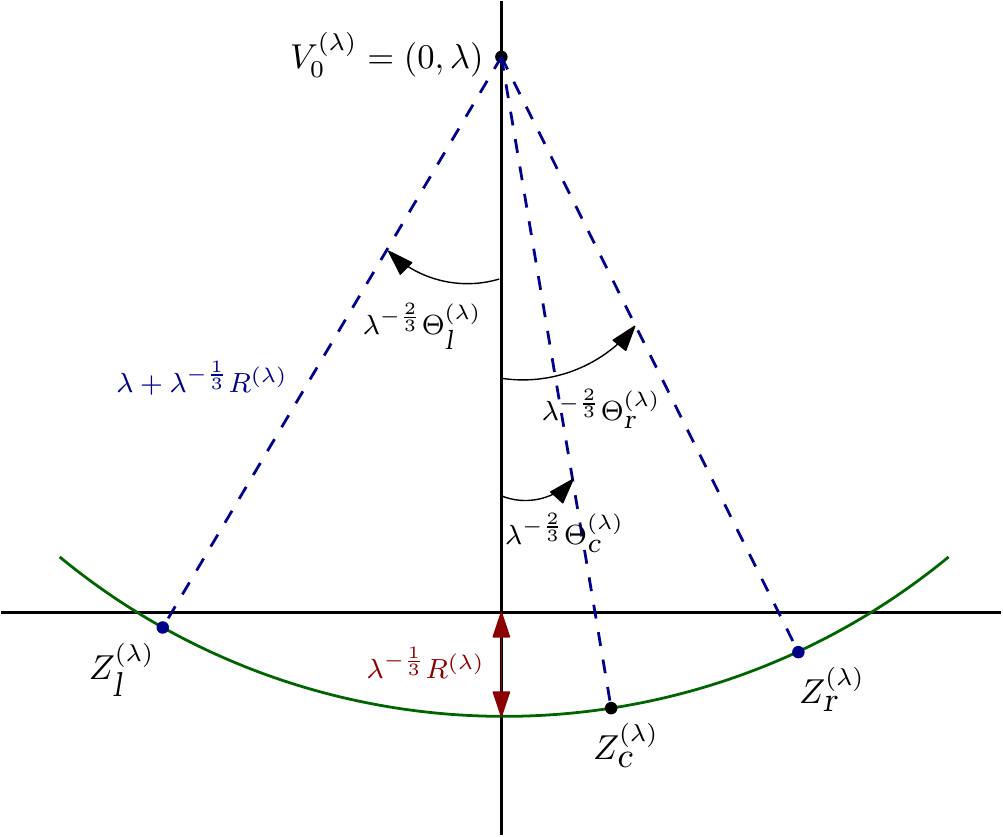}
\end{center}
\caption{The three nuclei $Z_{\textsl l}^{(\la)}$, $Z_{\textsl c}^{(\la)}$ and $Z_{\textsl r}^{(\la)}$ associated with the vertex $V_0^{(\la)}$ at height $\la$ and the quadruplet $(R^{(\la)},\Theta_{\textsl l}^{(\la)},\Theta_{\textsl c}^{(\la)},\Theta_{\textsl r}^{(\la)})$.}
\label{fig:Quadruplet}
\end{figure}

\begin{proof}
Let $f$ be a non-negative measurable function on the product space of $\R^2$ and the set of locally finite sets in $\bH^c$. For three points $x,y,z\in \bH^c$ which are not aligned, we denote by $D_{x,y,z}$ (resp. $c_{x,y,z}$) the unique disk containing the three points in its boundary (resp. the center of such disk) and by $\cP_{x,y,z}$ the set $\cP\cup\{x,y,z\}$. Moreover, we denote by $\cV(\cP)$ the point process of Voronoi vertices induced by $\cP$, by $\cP_{\ne}^3$ the set of all triplets of distinct points of $\cP$ and by $|\cdot|_2$ the two-dimensional Lebesgue measure. We recall that $c_{x,y,z}\in \cV(\cP_{x,y,z})$ if and only if $\cP\cap D_{x,y,z}=\emptyset$. Consequently, $\P[c_{x,y,z} \in \cV(\cP_{x,y,z})]=e^{-|D_{x,y,z}\cap {\bH^c}|_2}$. We then get 
\begin{align*}
\E\bigg[\sum_{v\in \cV(\cP)\cap \bH}f(v,\cP)\bigg]
& =\frac1{3!}\E\bigg[\sum_{(x,y,z) \in \cP_{\ne}^3}\indic_{\cV(\cP)\cap \bH}(c_{x,y,z})f(c_{x,y,z},\cP\setminus D_{x,y,z})\bigg] \\
& \hspace*{-1cm} = \frac16 \int_{(\bH^c)^3} \E\big[\indic_{\cV(\cP_{x,y,z})\cap \bH}(c_{x,y,z})f\big(c_{x,y,z},\cP_{x,y,z}\setminus D_{x,y,z}\big)\big]\dd x\dd y\dd z \\
& \hspace*{-1cm} = \int_{\{x_1<y_1<z_1\}} e^{-|D_{x,y,z}\cap {\bH^c}|_2}\E\big[f\big(c_{x,y,z},\cP_{x,y,z}\setminus D_{x,y,z}\big)\big]\indic_{\bH}(c_{x,y,z})\dd x\dd y\dd z
\end{align*}

\noi where we have used the independence of $\cP\cap D_{x,y,z}$ and $\cP\setminus D_{x,y,z}$. 

\pass Now we proceed with the following change of variables. Denoting by $\fhi=(\fhi_{\textsl l},\fhi_{\textsl c},\fhi_{\textsl r})$ and $\dd\fhi=\dd\fhi_{\textsl l}\dd\fhi_{\textsl c}\dd\fhi_{\textsl r}$, we apply the classical spherical Blaschke-Petkantschin formula
\begin{equation*}
\dd x\dd y\dd z = \rho^3\Delta\big(\tfrac{v_2}{\rho},\fhi\big)\dd \rho\dd\fhi\dd v
\end{equation*}
where $\rho$ is the radius of the disk $D_{x,y,z}\,$, $v=(v_1,v_2)\in\bH$ is the center $c_{x,y,z}$ and $\fhi_{\textsl l}$ (resp. $\fhi_{\textsl c},\fhi_{\textsl r}$) is the angle between the half-lines $c_{x,y,z}+\R_+e_2$ and $c_{x,y,z}+\R_+(x-c_{x,y,z})$ (resp. $c_{x,y,z}+\R_+(y-c_{x,y,z})$ and $c_{x,y,z}+\R_+(z-c_{x,y,z})$) and
\begin{align}\label{eq:defDelta}
\hspace{-0.251cm}\Delta\big(\tfrac{v_2}{\rho},\fhi\big)= 2
|\sin(\tfrac{\fhi_{\textsl c}-\fhi_{\textsl r}}{2})\sin(\tfrac{\fhi_{\textsl r}-\fhi_{\textsl l}}{2})
\sin(\tfrac{\fhi_{\textsl c}-\fhi_{\textsl l}}{2})|
\indic_{\{-\arccos(\frac{v_2}{\rho})<\fhi_{\textsl l}<\fhi_{\textsl c}<\fhi_{\textsl r}<\arccos(\frac{v_2}{\rho})\}\cap\{v_2<\rho\}}
\end{align}
is the area of the simplex spanned by the three unit vectors in directions $\fhi_{\textsl l},$ $\fhi_{\textsl c}$ and $\fhi_{\textsl r}$ respectively.

\pass Recalling that $x$, $y$ and $z$ are functions of $v$, $\rho$ and $\fhi$, we get
\begin{align*}
\E\bigg[\sum_{v\in \cV}f(v,\cP)\bigg]
= \int e^{-|D(v,\rho)\cap {\bH^c}|_2} \E\big[f(v,\cP_{x,y,z}\setminus D(v,\rho))\big] \rho^3\Delta\big(\tfrac{v_2}{\rho},\fhi\big)\dd \rho\dd\fhi\dd v.
\end{align*}

\noi Now we do the change of variables $\rho=v_2+v_2^{-\frac1{3}}r$ and $\fhi=v_2^{-\frac23}\theta$ where $\theta=(\theta_{\textsl l},\theta_{\textsl c},\theta_{\textsl r})$. We obtain

\begin{align}
\E\bigg[\sum_{v\in \cV}f(v,\cP)\bigg]
& = \int e^{-a(v_2,v_2+v_2^{-\frac1{3}}r)}\E\Big[f(v,\cP_{x,y,z}\setminus D(v,v_2+v_2^{-\frac1{3}}r))\Big]v_2^{-\frac73} \nonumber \\
& \hspace*{1.2cm}  \big(v_2+v_2^{-\frac13}r\big)^3\Delta\left(\mfrac{v_2}{v_2+v_2^{-\frac13}r},v_2^{-\frac23}\theta\right)\dd r\dd\theta \dd v \label{eq:formuleexplicitedensite}
\end{align}
where
\begin{align}
a(v_2,r) =|D(v,r)\cap \bH^c|_2 
= r^2 \left[\arccos\big(1-\tfrac{r-v_2}{r}\big)-\big(1-\tfrac{r-v_2}{r}\big)\big(\tfrac{2(r-v_2)}{r}-\tfrac{(r-v_2)^2}{r^2}\big)^{\frac12}\right].\label{eq:airecalotte}
\end{align}

\noi We deduce from \eqref{eq:formuleexplicitedensite} that for fixed $\la>0$, the probability density function of the quadruplet $(R^{(\la)},\Theta_{\textsl l}^{(\la)},\Theta_{\textsl c}^{(\la)},\Theta_{\textsl r}^{(\la)})$ is equal to
\begin{align}
\frac1{\delta(\la)}e^{-a(\la,\la+\la^{-\frac13}r)}\la^{-\frac73}(\la+\la^{-\frac13}r)^3
\Delta\bigg(\mfrac{\la}{\la+\la^{-\frac13}r},\la^{-\frac23}\theta\bigg)\label{eq:densityquadruplet}
\end{align}
where $\delta(\la)$ is the normalization factor. This completes the proof.
\end{proof}

\noi\textbf{Remark.} The function $\delta(\la)$ represents the density function at a point $(x,\la)$ of the intensity measure of the point process of vertices, i.e. 
\begin{equation*}\label{eq:densitesommets}
\delta(\la)= \int e^{-a(\la,\la+\la^{-\frac13}r)}\la^{-\frac73}(\la+\la^{-\frac13}r)^3\Delta\bigg(\mfrac{\la}{\la+\la^{-\frac13}r},\la^{-\frac23}
\theta\bigg)\dd r \dd\theta.
\end{equation*}
Estimating carefully each term in the integrand, we obtain that $\delta(\la)$ satisfies
\begin{equation}\label{eq:deltainfini}
\la^{\frac43}\delta(\la)\underset{\la\to\infty}{\sim} \delta(\infty)=2^{-\frac23}3^{-\frac43}\Gamma\left(\mfrac53\right).
\end{equation}

\subsection{The initial triangle}\label{sec:convpos}

\pass

\noi As announced earlier, the initial triangle is the triangle with vertices $V_0^{(\la)}$, $Z_{\textsl l}^{(\la)}$ and $Z_{\textsl c}^{(\la)}$. The points $Z_{\textsl l}^{(\la)}$ and $Z_{\textsl c}^{(\la)}$ are constructed from the quadruplet $(R^{(\la)},\Theta_{\textsl l}^{(\la)},\Theta_{\textsl c}^{(\la)},\Theta_{\textsl r}^{(\la)})$. In Proposition \ref{prop:simuquadruplet}, we prove the convergence in distribution and in total variation of the joint distribution of $(R^{(\la)},\Theta_{\textsl l}^{(\la)},\Theta_{\textsl c}^{(\la)},\Theta_{\textsl r}^{(\la)})$ and provide an explicit realization of the limiting distribution.

\begin{prop}\label{prop:simuquadruplet}
The quadruplet
$(R^{(\la)},\Theta_{\textsl l}^{(\la)},\Theta_{\textsl c}^{(\la)},\Theta_{\textsl r}^{(\la)})$ converges in distribution and in total variation to a quadruplet $(R,\Theta_{\textsl l},\Theta_{\textsl c},\Theta_{\textsl r})$ with probability density function
\begin{equation}\label{eq:palmdensity}
\frac{2^{-\frac43}3^{\frac43}}{\Gamma(\frac53)}\exp\Big(-\frac{4\sqrt{2}}{3}r^{\frac32}\Big)(\theta_{\textsl c} -\theta_{\textsl l})(\theta_{\textsl r} -\theta_{\textsl c})(\theta_{\textsl r}-\theta_{\textsl l})
\indicat_{\{-\sqrt{2r}<\theta_{\textsl l}<\theta_{\textsl c}<\theta_{\textsl r}<\sqrt{2r}\}}\indicat_{\{r>0\}}.
\end{equation}
Additionally, the following equality in distribution holds
\begin{equation}\label{eq:loilimquad}
(R,\Theta_{\textsl l},\Theta_{\textsl c},\Theta_{\textsl r})
\overset{(d)}{=}\left(\tfrac12\left(\tfrac{3}{2}G\right)^{\frac23},\left(\tfrac{3}{2}G\right)^{\frac13}U_{(1)},
\left(\tfrac{3}{2}G\right)^{\frac13}U_{(\eps)},\left(\tfrac{3}{2}G\right)^{\frac13}U_{(6)}\right)
\end{equation}
where:
\begin{itemize}[parsep=-0.15cm,itemsep=0.25cm,topsep=0.2cm,wide=0.15cm,leftmargin=0.5cm]
\item[-] the random variable $G\sim\mathrm{Gamma}(\tfrac83,1)$ is Gamma distributed,
\item[-] the random vector $(U_{(1)},\ldots,U_{(6)})$ is the order statistics of a random vector $(U_1,\ldots,U_6)$ uniformly distributed in $(-1,1)^6$ and independent of $G$,
\item[-] the random variable $\eps$ is independent of $G$ and $(U_1,\ldots,U_6)$ and equal to $3$ or $4$ with probability $\frac12$.
\end{itemize}
\end{prop}

\begin{proof}
\noi We look for the asymptotics of each factor of the integrand in the right-hand side of \eqref{eq:densitesommets}. We deduce from \eqref{eq:airecalotte} that 
\begin{align*} 
a(\la,\la+\la^{-\frac13}r)
&\underset{\la\to\infty}{\sim}\frac{4\sqrt2}{3}r^{\frac32}.
\end{align*}

\noi Thanks to \eqref{eq:defDelta}, we obtain
\begin{align*} 
\la^{-\frac73}(\la+\la^{-\frac13}r)^3\Delta\big(\tfrac{\la}{\la+\la^{-\frac13}r},\la^{-\frac23}\theta\big)\underset{\la\to\infty}{\sim} \mfrac14 \la^{-\frac43}
(\theta_{\textsl c} -\theta_{\textsl l})(\theta_{\textsl r} -\theta_{\textsl c})(\theta_{\textsl r}-\theta_{\textsl l})
\indicat_{\{-\sqrt{2r}<\theta_{\textsl l}<\theta_{\textsl c}<\theta_{\textsl r}<\sqrt{2r}\}}.
\end{align*}

\noi This and \eqref{eq:deltainfini} prove that the probability density function given at \eqref{eq:densityquadruplet} converges almost everywhere when $\la\to\infty$ to the expression at \eqref{eq:palmdensity}. It remains to apply Scheff\'e's Lemma to get the required convergence in distribution and in total variation.

\pass Let us show the realization of the quadruplet $(R,\Theta_{\textsl l},\Theta_{\textsl c},\Theta_{\textsl r})$. We deduce from \eqref{eq:palmdensity} that the probability density function of the distribution of $R$ is proportional to $\exp(-\tfrac{4\sqrt{2}}{3}r^{\frac32}) r^3$. Noticing that for any non-negative measurable function $g$ on $\R_+$,
\begin{equation*}
\int_0^{\infty} g(r)\exp(-\tfrac{4\sqrt{2}}{3}r^{\frac32})r^3\dd r\propto \int_0^{\infty} g(\tfrac1{2}(\tfrac32 u)^{\frac23})e^{-u}u^{\frac53}\dd u,
\end{equation*}
we obtain that $R\overset{(d)}{=} \tfrac12\left(\tfrac{3}{2}G\right)^{\frac23}$. Moreover, we deduce from \eqref{eq:palmdensity} that the variable $R$ is independent of $(2R)^{-\frac12}(\Theta_{\textsl l},\Theta_{\textsl c},\Theta_{\textsl r})$.

\pass It then remains to show that the distribution of $(U_{(1)},U_{(\varepsilon)},U_{(6)})$ has a probability density function proportional to $(v-u)(w-v)(w-u)\indicat_{\{-1<u<v<w<1\}}$. For any non-negative measurable function $h$ on $(-1,1)^3$ we get
\begin{align}\label{eq:simuu1}
\E\big[h(U_{(1)},U_{(3)},U_{(6)})\big]
& = 6!\int h(u_1,u_3,u_6)\indicat_{\{-1<u_1<\cdots<u_6<1\}}\dd u_1\cdots \dd u_6 \nonumber\\
& = 6!\int h(u,v,w)(v-u)(w-v)^2\indicat_{\{-1<u<v<w<1\}}\dd u\dd v \dd w.
\end{align}

\noi In the same way,
\begin{align}\label{eq:simuu2}
\E\big[h(U_{(1)},U_{(4)},U_{(6)})\big]= 6! \int h(u,v,w)(v-u)^2(w-v)\indicat_{\{-1<u<v<w<1\}}\dd u\dd v \dd w.
\end{align}

\noi Writing the half-sum of \eqref{eq:simuu1} and \eqref{eq:simuu2}, we obtain the required result.
\end{proof}

\noi We are now ready to introduce the first isosceles triangle with vertices $V_0^{(\la)}=(0,\la)$, $Z_{\textsl l}^{(\la)}$ and $Z_{\textsl c}^{(\la)}$ and the two associated lengths. Let us define the two quantities $B_0^{(\la)}$ and $H_0^{(\la)}$ as the normalized lengths of, respectively, the half-basis and the height of that triangle, i.e.
\begin{equation}\label{eq:defBoHo}
B_0^{(\la)}=\mfrac{\la^{-\frac13}}{2}\|Z_{\textsl c}^{(\la)}-Z_{\textsl l}^{(\la)}\| \quad\text{and}\quad H_0^{(\la)}=\la^{-1}\dd\big(V_0^{(\la)},Z_{\textsl c}^{(\la)}+\R(Z_{\textsl c}^{(\la)}-Z_{\textsl l}^{(\la)})\big)
\end{equation}
where $\|\cdot\|$ is the Euclidean norm and $\dd(\cdot,\cdot)$ is the Euclidean distance in $\R^2$.

\noi Applying the cosine law in the triangle with vertices $V_0^{(\la)}$, $Z_{\textsl l}^{(\la)}$ and $Z_{\textsl c}^{(\la)}$ we get
\begin{equation}\label{eq:B0H0}
(B_0^{(\la)},H_0^{(\la)})=(\la+\la^{-\frac13}R^{(\la)})
\bigg(\la^{-\frac13}\tan\big(\mfrac{\la^{-\frac23}}{2}(\Theta_{\textsl c}^{(\la)}-\Theta_{\textsl l}^{(\la)})\big),\la^{-1}
\cos\big(\mfrac{\la^{-\frac23}}{2}(\Theta_{\textsl c}^{(\la)}-\Theta_{\textsl l}^{(\la)})\big)\bigg).
\end{equation}

\noi Therefore, Taylor expanding the expression above and using that $(R^{(\la)},\Theta_{\textsl l}^{(\la)},\Theta_{\textsl c}^{(\la)},\Theta_{\textsl r}^{(\la)})$ converges in distribution thanks to Proposition \ref{prop:simuquadruplet}, we obtain that the couple $(B_0^{(\la)},H_0^{(\la)})$ converges in distribution to an explicit limit given in Lemma \ref{lem:convBoHo} below. 

\begin{lem}\label{lem:convBoHo}
When $\la\longrightarrow\infty$, the following convergence holds:
\begin{equation}\label{eq:convBoHo}
(B_0^{(\la)},H_0^{(\la)})\overset{(d)}{\longrightarrow}\big(\mfrac12(\Theta_{\textsl c}-\Theta_{\textsl l}),1\big).
\end{equation}
\end{lem}

\subsection{The subsequent triangles}\label{sec:strategyproof}

\pass 

\noi We recall that $\{V_n^{(\la)}\}_n$ is the set of consecutive Voronoi vertices with non-negative second coordinate that are reached when walking along the left branch $\cB^{(\la)}$ by choosing the right direction at each intersection. The sequence $\{Z_n^{(\la)}\}_n$ is the set of consecutive Voronoi nuclei such that for $n\ge 1$, $V_n^{(\la)}$ belongs to the three Voronoi cells associated with $Z_{n-1}^{(\la)}$, $Z_n^{(\la)}$ and $Z_{\textsl c}^{(\la)}$. The subsequent triangles are then the triangles with vertices $V_n^{(\la)}$, $Z_{n}^{(\la)}$ and $Z_{\textsl c}^{(\la)}$.

\pass For $n\ge \cN(\cB^{(\la)})$ where we recall that $\cN(\cB^{(\la)})$ coincides with the index of the first $V_n^{(\la)}$ in $\bH^c$, we fix $V_n^{(\la)}=V_{\cN(\cB^{(\la)})}^{(\la)}$. For all $n<\cN(\cB^{(\la)})$, we introduce the two quantities $B_n^{(\la)}$ and $H_n^{(\la)}$ as the normalized lengths of, respectively, the half-basis and the height of the triangle with vertices $V_n^{(\la)}$, $Z_n^{(\la)}$ and $Z_{\textsl c}^{(\la)}$, i.e.
\begin{equation}\label{eq:defBnHn}
B_n^{(\la)}=\mfrac{\la^{-\frac13}}{2}\|Z_n^{(\la)}-Z_{\textsl c}^{(\la)}\| \quad\text{and}\quad H_n^{(\la)}=\la^{-1}\dd\big(V_n^{(\la)},Z_{\textsl c}^{(\la)}+\R(Z_n^{(\la)}-Z_{\textsl c}^{(\la)})\big).
\end{equation}
In particular \eqref{eq:defBnHn} extends \eqref{eq:defBoHo} with the convention $Z_{\textsl l}^{(\la)}=Z_0^{(\la)}$.

\begin{prop}\label{prop:cestmarkov}
The sequence $\{(B_n^{(\la)},H_n^{(\la)})\}_n$ coincides for $n < \cN(\cB^{(\la)})$ with a  homogeneous Markov chain whose transition probability $p^{(\la)}$ is given by \eqref{eq:formuleexactpla}. 
\end{prop}

\begin{proof}
We start by computing explicitly the exact transition probability of the Markov chain $\{(B_n^{(\la)},H_n^{(\la)})\}_n$ once the point $Z_{\textsl c}^{(\la)}=z_{\textsl c}$ is fixed. Conditional on the event $\{(B_n^{(\la)},H_n^{(\la)})=(b,h)\}$ where $(b,h)$ is given, the points $Z_n^{(\la)}$ and $V_n^{(\la)}$ are fully determined, up to a rotation around $z_{\textsl c}$. Assuming for instance that the line containing $z_{\textsl c}$ and $Z_n^{(\la)}$ is horizontal, we get
\begin{equation*}
Z_n^{(\la)}=z_n=z_{\textsl c}-2b e_1\quad\text{and}\quad V_n^{(\la)}=z_{\textsl c}-b e_1-h e_2.
\end{equation*}
Moreover, the point process $\cP^{(\la)}$ is distributed as $\cP^{(\la)}_{(b,h)}$ which is the union of $\{z_{\textsl c},z_n\}$ and of a homogeneous Poisson point process in $\bH^c$ and outside of the disk centered at $v_n$  and containing $z_{\textsl c}$ and $z_n$ on its boundary.

\pass Let us define for any $z'\in \bH^c$ the function
\begin{equation}
G^{(\la)}_{(b,h)}(z')=\bigg(\mfrac{\la^{-\frac13}}{2}\|z'-z_{\textsl c}\|,\la^{-1}\dd(c_{z_{\textsl c},z_n,z'},(z_{\textsl c}+\R(z'-z_{\textsl c})))\bigg)
\end{equation}
where we recall that $c_{x,y,z}$ is the center of the unique circle containing the three points $x,y,z\in\R^2$. In other words, $G^{(\la)}_{(b,h)}(z')$ is the new couple $(b',h')$ corresponding to the triangle with vertices $z_{\textsl c}$, $z_n$ and $z'$.

\pass It follows from Mecke-Slivnyak's formula that for any non-negative measurable function $f$,
\begin{align}
\E\Big[f({B_{n+1}^{(\la)}},{H_{n+1}^{(\la)}})\,\Big|\ (B_n^{(\la)},H_n^{(\la)})=(b,h) \Big ]
& = \E\Bigg[\sum_{z'\in\cP^{(\la)}_{(b,h)}}f(G^{(\la)}_{(b,h)}(z'))\indic_{\cV}(c_{z_{\textsl c},z_n,z'})\Bigg] \nonumber\\
& = \int f(G^{(\la)}_{(b,h)}(z'))\P\big[\cP^{(\la)}_{(b,h)}\cap \Delta_n^{(\la)} = \emptyset\big]\dd z' \nonumber\\
& = \int f(G^{(\la)}_{(b,h)}(z'))e^{-|\Delta_n^{(\la)}|_2}\dd z'\label{eq:applimeckepourp}
\end{align}
where $\Delta_n^{(\la)}$ is the set
\begin{equation*}
\Delta_n^{(\la)}=D_{z_{\textsl c},z_n,z'}\setminus D_{z_{\textsl c},z_{n-1},z_n}.
\end{equation*}

\noi We now wish to apply two consecutive changes of variables. The first one consists in replacing $z'$ by the couple $(b',h')=G^{(\la)}_{(b,h)}(z')$. We first show that $(b',h')\in S_{(b,h)}$ where $S_{(b,h)}$ is the set
\begin{equation}\label{eq:domainesautorisesbh}
S_{(b,h)}=\left\{(b',h'): 0<b'<b\text{ and }\la^{-\frac23}(b^2-b'^2)^{\frac12}<h'<(h^2+\la^{-\frac43}(b^2-b'^2))^{\frac12}\right\}.
\end{equation}

\noi Indeed, since $z'$ lies on the arc between $z_{\textsl c}$ and $z_n$, the distance $b'=\|z'-z_{\textsl c}\|$ is less than $b=\|z_n-z_{\textsl c}\|$. Moreover, the radius of the disk $D_{z_{\textsl c},z_n,z'}$ is equal to $(\la^2 h'^2+\la^{\frac23}b'^2)^{\frac12}$ and is less than the radius of the disk $D_{z_{\textsl c},z_{n-1},z_n}$, equal to $(\la^2 h^2+\la^{\frac23}b^2)^{\frac12}$, which explains the upper bound for $h'$. We refer to Figure \ref{fig:FirstTriangles} for the case $n=0$. Finally, the radius of $D_{z_{\textsl c},z_n,z'}$ is larger than $\la^{\frac13}b$ since it is the length of the hypotenuse $[c_{z_{\textsl c},z_n,z'},z_{\textsl c}]$ of a right-triangle with one side of length $\la^{\frac13}b$. This justifies the lower bound for $h'$. 

\pass The next change of variable consists in writing the polar coordinates $(\rho',\theta')$ of $z'$ with respect to the origin at $z_{\textsl c}$ and the $x$-axis $z_{\textsl c}+\R(z_n-z_{\textsl c})$ in function of the couple $(b',h')$:
\begin{align*}
\Bigg\{\begin{array}{ll}
\rho' & = 2 \la^{\frac13}b' \\
\theta' & = \arctan\Big(\mfrac{\la^{\frac13}b}{(\la^2h'^2-\la^{\frac23}(b^2-b'^2))^{\frac12}}\Big)-\arctan\big(\mfrac{\la^{\frac13}b'}{\la h'}\big)
\end{array}.
\Bigg.
\end{align*}

\noi In particular, we deduce that the corresponding Jacobian can be written as
\begin{align}\label{eq:vraijacob}
\frac{\dd z'}{\dd b'\dd h'}=\frac{\rho'\dd \rho'\dd \theta'}{\dd b'\dd h'}
= \mfrac{4b'}{h'^2+\la^{-\frac43}b'^2}\big(b\big(1-\la^{-\frac43}\mfrac{b^2-b'^2}{h'^2}\big)^{-\frac12}-b'\big).
\end{align}

\noi We now turn our attention to the calculation of the area  $|\Delta_n^{(\la)}|_2$ which only depends on $b$, $h$, $b'$ and $h'$ and not on the actual positions of the three points $z_{\textsl c}$, $z_n$ and $v_n$. The set $\Delta_n^{(\la)}$ is nothing but the difference between two caps. In our setting, each circular cap is coded by the half-length $\ell$ of the cap basis and by the distance $L$ from the center of the circle to the cap basis. The area of such cap is given by the formula
\begin{equation*}\label{eq:caparea}
(\ell^2+L^2)\big(\arctan (\mfrac{\ell}{L})-\mfrac{\ell L}{\ell^2+L^2}\big).
\end{equation*}
In our situation, we apply the above formula to the two couples $(\ell=\la^{\frac13}b, L=\la h)$ and  $(\ell=\la^{\frac13}b, L=(\la^2 h'^2-\la^{\frac23}(b^2-b'^2))^{\frac12})$, see the set $\Delta_0^{(\la)}$ between the purple and green circles in Figure \ref{fig:FirstTriangles}. Consequently, we get
\begin{align}
|\Delta_n^{(\la)}|_2 & = (\la^2h'^2+\la^{\frac23}b'^2)\big(\arctan\Big(\mfrac{\la^{\frac13}b}{(\la^2h'^2-\la^{\frac23}(b^2-b'^2))^{\frac12}}\Big) \nonumber \\
& \hspace{0.5cm} - \mfrac{\la^{\frac13}b}{\la^2h'^2+\la^{\frac23}b'^2}\big(\la^2h'^2-\la^{\frac23}(b^2-b'^2)\big)^{\frac12}\big) \nonumber \\
& \hspace{1cm} - (\la^2h^2+\la^{\frac23}b^2)\big(\arctan\big(\la^{-\frac23}\mfrac{b}{h}\big)-\la^{-\frac23}\mfrac{b}{h}
\big(1+\la^{-\frac43}\mfrac{b^2}{h^2}\big)^{-1}\big).
\label{eq:formuleairecap}
\end{align}

\noi Combining \eqref{eq:applimeckepourp} and \eqref{eq:vraijacob}, we obtain the following explicit formula for the transition probability, i.e.
\begin{equation}\label{eq:formuleexactpla}
p^{(\la)}((b,h),(b',h')) = \mfrac{4b'}{h'^2+\la^{-\frac43}b'^2}\big(b\big(1-\la^{-\frac43}\mfrac{b^2-b'^2}{h'^2}\big)^{-\frac12}-b'\big)e^{-|\Delta_n^{(\la)}|_2}\indicat_{S_{(b,h)}}(b',h')
\end{equation}
where $|\Delta_n^{(\la)}|_2$ is given at \eqref{eq:formuleairecap}.
\end{proof}

\subsection{How to use the triangles to prove the limit shape result?}\label{sec:bullet}

\pass

\noi We explain now how to study the limit of the Markov chain induced by the sequence of triangles defined in Section \ref{sec:strategyproof} and how to deduce from it the convergence of the renormalized Voronoi cell to the limiting menhir $\cC^{(\infty)}$.
\begin{itemize}[parsep=-0.15cm,itemsep=0.45cm,topsep=0.25cm,wide=0.15cm,leftmargin=0.5cm]
\item In Lemma \ref{lem:lessommetsretrouves} below, we write the sequence of vertices $\{V_n^{(\la)}\}_n$ starting at $V_0^{(\la)}=(0,\la)$ as a function of the Markov sequence 
$\{(B_n^{(\la)},H_n^{(\la)})\}_n$
and of the initial angle $\frac12(\Theta_{\textsl c}^{(\la)}+\Theta_{\textsl r}^{(\la)})$ between the vertical direction and the line from $V_0^{(\la)}$ to $Z_{\textsl r}^{(\la)}$. Note that the initial state $(B_0^{(\la)},H_0^{(\la)})$ is itself a functional of $(R^{(\la)},\Theta_{\textsl c}^{(\la)},\Theta_{\textsl r}^{(\la)})$ by \eqref{eq:B0H0}.

\item In Section \ref{sec:coupling}, we show that it is possible to couple the Markov chain $\{(B_n^{(\la)},H_n^{(\la)})\}_n$ with transition $p^{(\la)}$ with a Markov chain $\{(B_n^{(\infty)},H_n^{(\infty)})\}_n$  whose transition $p^{(\infty)}$ is the pointwise limit of $p^{(\la)}$, both starting with the same initial distribution.

\item In Section \ref{sec:mainproof}, we apply the affine transformation $F^{(\la)}$ to the sequence $\{V_n^{(\la)}-Z_{\textsl c}^{(\la)}\}_n$ and we proceed to a Taylor expansion in the formulas of Lemma \ref{lem:lessommetsretrouves} which gives the limit of the branch $\cB^{(\la)}$ and subsequently of the Voronoi cell $\cC^{(\la)}$. The technique is based on a mixture of the coupling arguments from Section \ref{sec:coupling} and some stability properties of the Markov chain $\{(B_n^{(\infty)},H_n^{(\infty)})\}_n$.
\end{itemize}

\noi Lemma \ref{lem:lessommetsretrouves} is of geometric nature and makes explicit the sequence of vertices $\{V_n^{(\la)}\}_n$. For any $x,y,z\in\R^2$, we denote by $\angle (xyz)$ the angle between the two half-lines $y+\R_+(x-y)$ and $y+\R_+(z-y)$.

\begin{lem}\label{lem:lessommetsretrouves}
The sequence $\{V_n^{(\la)}\}_n$ is defined recursively by 
$V_0^{(\la)}=(0,\la)$, $\angle (oV_0^{(\la)}V_1^{(\la)})=\frac12(\Theta_{\textsl c}^{(\la)}+\Theta_{\textsl l}^{(\la)})$ and the following identities:
\begin{equation}\label{eq:dist2somments}
\la^{-1}\|V_n^{(\la)}-V_{n+1}^{(\la)}\| = H_n^{(\la)}-\big((H_{n+1}^{(\la)})^2-\la^{-\frac43}((B_n^{(\la)})^2-(B_{n+1}^{(\la)})^2)\big)^{\frac12}
\end{equation}
and
\begin{equation}\label{eq:angle2somments}
\angle (V_n^{(\la)} V_{n+1}^{(\la)} V_{n+2}^{(\la)}) 
= \arcsin \bigg(\la^{-\frac23}\frac{B_n^{(\la)}}{\big((H_{n+1}^{(\la)})^2+\la^{-\frac43}(B_{n+1}^{(\la)})^2\big)^{\frac12}}\bigg)-\arctan\bigg(\la^{-\frac23}\frac{{B_{n+1}^{(\la)}}}{{H_{n+1}^{(\la)}}}\bigg). 
\end{equation}
\end{lem}

\begin{proof} The angle $\angle (oV_0^{(\la)}V_1^{(\la)})$ is the angle between the vertical direction and the bisecting line of the segment $[Z_{\textsl c}^{(\la)},Z_{\textsl l}^{(\la)}]$, i.e. $\frac12 (\Theta_{\textsl c}^{(\la)}+\Theta_{\textsl l}^{(\la)})$. Let $A_n^{(\la)}$ be the projection of $V_n^{(\la)}$ onto the line joining $Z_{\textsl c}^{(\la)}$ and $Z_n^{(\la)}$. Applying Pythagora's Theorem twice, we get
\begin{align*}
\|V_n^{(\la)}-V_{n+1}^{(\la)}\| & =
\la H_n^{(\la)}-\|V_{n+1}^{(\la)}-A_n^{(\la)}\| \\
& =\la H_n^{(\la)}-\big(\|Z_ {\textsl c}^{(\la)}-V_{n+1}^{(\la)}\|^2-\la^{\frac23}(B_n^{(\la)})^2\big)^{\frac12} \\
& = \la H_n^{(\la)}-\big(((\la H_{n+1}^{(\la)})^2+\la^{\frac23}(B_{n+1}^{(\la)})^2)-(B_n^{(\la)})^2     \big)^{\frac12}
\end{align*}
which implies \eqref{eq:dist2somments}.

\noi Moreover,
\begin{align*}
\angle(V_n^{(\la)}V_{n+1}^{(\la)}V_{n+2}^{(\la)})
& = \angle(Z_{\textsl c}^{(\la)}V_{n+1}^{(\la)}V_n^{(\la)})-\angle(Z_{\textsl c}^{(\la)}V_{n+1}^{(\la)}V_{n+2}^{(\la)}) \nonumber 
\end{align*}
which gives \eqref{eq:angle2somments} and completes the proof.
\end{proof}

\addtocontents{toc}{\vspace{0.25cm}}%
\section{Coupling with an idealized Markov chain}\label{sec:coupling}

\noi In this section, we introduce a Markov chain $\{(B_n^{(\infty)},H_n^{(\infty)})\}_n$ whose transition kernel is the pointwise limit of the transition kernel of the Markov chain $\{(B_n^{(\la)},H_n^{(\la)})\}_n$. Section \ref{sec:probarepres} is devoted to a nice probabilistic representation of this new Markov chain while Section \ref{sec:couplage} deals with an explicit coupling of the two Markov chains on a good event of high probability. 

\subsection{An idealized Markov chain and its probabilistic representation}\label{sec:probarepres}

\pass 

\noi Proposition \ref{prop:egaliteloiidealisee} below provides the explicit kernel of $\{(B_n^{(\infty)},H_n^{(\infty)})\}_n$ and a useful representation of that Markov chain.

\begin{prop}\label{prop:egaliteloiidealisee}
When $\la\longrightarrow\infty$, the transition kernel $p^{(\la)}((b,h),(b',h'))$ converges pointwise in $(0,\infty)^2$ to $p^{(\infty)}((b,h),(b',h'))$ where
\begin{equation}\label{eq:transideale}
p^{(\infty)}((b,h),(b',h')) = \mfrac{4(b-b')b'}{h'^2}\exp\big(-\mfrac23b^3\big(\mfrac1{h'}-\mfrac1{h}\big)\big)\indicat_{(0,b)\times(0,h)}(b',h').
\end{equation}
Moreover, a representation of the Markov chain with transition kernel $p^{(\infty)}$ is given by the sequence $\{(B_n^{(\infty)},H_n^{(\infty)})\}_n$ defined by the almost sure induction equality
\begin{equation}\label{eq:BnHnrec}
(B_{n+1}^{(\infty)},H_{n+1}^{(\infty)})=\Bigg(\beta_n B_n^{(\infty)},\frac{(B_n^{(\infty)})^3H_n^{(\infty)}}{(B_n^{(\infty)})^3+\tfrac32\xi_n H_n^{(\infty)}}\Bigg)
\end{equation}
where $\{\beta_n\}_n$ and $\{\xi_n\}_n$ are two independent sequences of independent and identically distributed random variables such that $\beta_n\sim\mathrm{Beta}(2,2)$ and $\xi_n\sim\mathrm{Exp}(1)$ respectively.

\end{prop}

\begin{proof} First we show the convergence of $p^{(\la)}$ to the limit transition $p^{(\infty)}$ given at \eqref{eq:transideale}. To do so, it is enough to get the asymptotics in \eqref{eq:domainesautorisesbh}, \eqref{eq:vraijacob} and \eqref{eq:formuleairecap}. They are straightforward in  \eqref{eq:domainesautorisesbh} and \eqref{eq:vraijacob}. As for the area of $\Delta_n^{(\la)}$, Taylor expanding \eqref{eq:formuleairecap} leads to
\begin{align}\label{eq:limdeltan}
\lim_{\la\to\infty}|\Delta_n^{(\la)}|_2=\mfrac23b^3\big(\mfrac1{h'}-\mfrac1{h}\big).
\end{align}
The result follows. Next, let us observe that the transition probability $p^{(\infty)}$ given by \eqref{eq:transideale} may be factorized as
\begin{equation}\label{eq:factorpinfty}
p^{(\infty)}((b,h),(b',h'))=\left(\mfrac{6(b-b')b'}{b^3}\indicat_{(0,b)}(b')\right)\bigg(\mfrac23\mfrac{b^3}{h'^2}\exp\big(-\mfrac23 b^3\big(\mfrac1{h'}-\mfrac1{h}\big)\big)\indicat_{(0,h)}(h')\bigg)
\end{equation}
where the first (resp. the second) expression inside brackets is the probability density function of ${B_{n+1}^{(\infty)}}$ (resp. ${H_{n+1}^{(\infty)}}$) conditional on $\big\{(B_n^{(\infty)},H_n^{(\infty)})=(b,h)\big\}$.

\pass Let us show the recursion formula satisfied by the sequence $\{B_n^{(\infty)}\}_n$. Conditional on $\{B_n^{(\infty)}=b\}$, the probability density function of ${B_{n+1}^{(\infty)}}$ is $\mfrac{6(b-b')b'}{b^3}\indicat_{(0,b)}(b')$, which in turn implies that $\frac1{b}{B_{n+1}^{(\infty)}}$ admits a probability density function given by $6t(1-t)\indic_{(0,1)}(t)$, i.e. follows a $\mathrm{Beta}(2,2)$ distribution. In other words, this proves that the sequence $\{B_n^{(\infty)}\}_n$ satisfies the recursion relation
\begin{equation}\label{eqrecbn}
{B_{n+1}^{(\infty)}}=\beta_n B_n^{(\infty)}
\end{equation}
where $\{\beta_n\}_n$ is a sequence of iid random variables with common distribution $\mathrm{Beta}(2,2)$.

\pass Next, the second factor in the right-hand side of \eqref{eq:factorpinfty} provides the probability density function of ${H_{n+1}^{(\infty)}}$ conditional on $\{(B_n^{(\infty)},H_n^{(\infty)})=(b,h)\}$. We now use the trick of calculating the associated cumulative distribution function and inverse it. More precisely, the cumulative density function $G_n(\cdot\,|\,(b,h))$ of ${H_{n+1}^{(\infty)}}$ conditional on $\{(B_n^{(\infty)},H_n^{(\infty)})=(b,h)\}$ is written
\begin{align*}
G_n(t\,|\,(b,h)) = \frac23 b^3\int_0^t \frac1{h'^2}\exp\left(-\frac23b^3\left(\frac1{h'}-\frac1{h}\right)\right) \dd h'= \exp\left(-\frac23 b^3\left(\frac1{t}-\frac1{h}\right)\right).
\end{align*}

\noi Therefore, conditional on $\{(B_n^{(\infty)},H_n^{(\infty)})=(b,h)\}$, the following holds:
\begin{equation*}
G_n({H_{n+1}^{(\infty)}}\,|\,(b,h)) = \exp\bigg(-\frac23 b^3\bigg(\frac1{{H_{n+1}^{(\infty)}}}-\frac1{h}\bigg)\bigg) \overset{(d)}{=} U
\end{equation*}
where $U$ is a variable uniformly distributed in $(0,1)$.

\noi This implies that conditional on $\{(B_n^{(\infty)},H_n^{(\infty)})=(b,h)\}$,
\begin{equation*}
-\frac23b^3\bigg(\frac1{{H_{n+1}^{(\infty)}}}-\frac1{h}\bigg) \overset{(d)}{=} \log U. 
\end{equation*}

\noi Now deconditioning the relation above, we obtain that
\begin{equation}\label{eq:relrecpratiquepourplustard}
\frac23(B_n^{(\infty)})^3\bigg(\frac1{{H_{n+1}^{(\infty)}}}-\frac1{H_n^{(\infty)}}\bigg)=\xi_n 
\end{equation}
is an exponential variable with mean one. This proves in turn that the sequence $\{H_n^{(\infty)}\}_n$ satisfies the recursive relation
\begin{equation*}\label{eqrectn}
H_{n+1}^{(\infty)}=\frac{(B_n^{(\infty)})^3 H_n^{(\infty)}}{(B_n^{(\infty)})^3+\frac32\xi_n H_n^{(\infty)}}.
\end{equation*} 
This completes the proof.
\end{proof}

\noi The next corollary is a direct consequence of the probabilistic representation of $\{B_n^{(\infty)}\}_n$ that will be used in the construction of the coupling in Section \ref{sec:couplage}.

\begin{cor}\label{cor:logBn}
The sequence $\{B_n^{(\infty)}\}_n$ decreases almost surely to $0$. More precisely, almost surely:
\begin{equation*}
-\frac1{n}\log B_n^{(\infty)}\underset{n\to\infty}{\longrightarrow}\frac56.
\end{equation*}
\end{cor}

\begin{proof} Thanks to Proposition \ref{prop:egaliteloiidealisee} we get $B_n^{(\infty)}=B_0^{(\infty)}\beta_0\beta_1\cdots\beta_{n-1}$. It then remains to notice that
\begin{equation*}
\E(-\log\beta_0)=\int_0^1 6t(1-t)\log t \dd t = \frac{5}{6}
\end{equation*}
and use the law of large numbers.
\end{proof}

\noi In Proposition \ref{prop:loiTn} below we introduce an auxiliary sequence $\{T_n^{(\infty)}\}_n$ of so-called {\it shape characteristics} of the consecutive triangles associated with the Markov chain $\{(B_n^{(\infty)},H_n^{(\infty)})\}_n$. One of the assets of that sequence is that $T_n^{(\infty)}$ converges in distribution to an explicit limit law without the need of any renormalization.  

\begin{prop}\label{prop:loiTn}
The sequence $\{T_n^{(\infty)}\}_n$ defined by
\begin{equation}\label{eq:defTninfty}
T^{(\infty)}_n=\frac{(B_n^{(\infty)})^3}{H_n^{(\infty)}}
\end{equation}
satisfies the almost sure induction relation
\begin{equation}\label{eq:Tnrec}
T_{n+1}^{(\infty)}=\beta_n^3\big(T_n^{(\infty)}+\tfrac32\xi_n\big)
\end{equation}
where $\{\beta_n\}_n$ and $\{\xi_n\}_n$ are two independent sequences of independent and identically distributed random variables such that $\beta_n\sim\mathrm{Beta}(2,2)$ and $\xi_n\sim\mathrm{Exp}(1)$ respectively.
Moreover, the sequence $\{T_n^{(\infty)}\}_n$ converges in distribution to a random variable $T$ which can be expressed as
\begin{equation}\label{eq:loiTinfty}
T=\mfrac32 T_1^3 T_2 T_3
\end{equation}
where $T_1$, $T_2$ and $T_3$ are independent random variables such that $T_1\sim\mathrm{Beta}(2,2)$, $T_2\sim\mathrm{Gamma}(\frac53,1)$ and $T_3\sim\mathrm{Beta}(2,\frac23)$ respectively.
\end{prop}

\begin{proof} The fact that $\{T^{(\infty)}_n\}_n$ satisfies the recursive relation \eqref{eq:Tnrec} is straightforward thanks to \eqref{eqrecbn} and \eqref{eq:relrecpratiquepourplustard}. Let us first show that the sequence $\{T^{(\infty)}_n\}_n$ converges in distribution. We use Letac's principle in \cite{Let}, i.e. we introduce the sequence $\{T_n'\}_n$ such that $T_0'=T^{(\infty)}_0$ and 
\begin{equation*}
T_n'=\beta_1^3\big(\beta_2^3\big(\cdots \big(\beta_{n-2}^3\big(\beta_{n-1}^3\big(T_0'+\tfrac32\xi_{n-1}\big)+\tfrac32\xi_{n-2}\big)+\cdots\big)+\tfrac32\xi_1\big).
\end{equation*}

\noi We get that $\{T_n'\}_n$ converges almost surely to $\frac32\sum_{n=1}^{\infty}\beta_1^3\cdots \beta_n^3\xi_n$ which is itself convergent by a direct use of the Three-series Theorem. Since $T_n^{(\infty)}$ and $T_n'$ have the same distribution, the sequence $\{T_n^{(\infty)}\}_n$ converges in distribution to a certain random variable $T$.

\pass Let us now identify the distribution of $T$. We get from \eqref{eq:Tnrec} that 
\begin{equation}\label{eq:equationfonc}
T\overset{(d)}{=} \beta^3 \big(T + \tfrac32 \xi \big)    
\end{equation}
with $\beta\sim\mathrm{Beta}(2,2)$, $\xi\sim \mathrm{Exp}(1)$ and $\beta$, $\xi$, $T$ independent. Thanks to \cite[Th. 1.5 (i)]{Ver79}, the solution of \eqref{eq:equationfonc} is unique.

\pass Determination of solutions of the random difference equation \eqref{eq:Tnrec} may be done in the same spirit as the resolution of some Beta-Gamma convolutions studied in \cite{DDu}. More precisely we are looking for a solution of \eqref{eq:equationfonc} which has finite moments of any order and we intend to calculate explicitly these moments. To do so, let us set $c=\frac32$ and for every $k\ge 0$, let us introduce  
\begin{equation*}
t_k=\frac{1}{c^k}\frac{1}{k!}\E[T^k].
\end{equation*}

\pass We recall two general facts that will be used again at the end of the proof: if $X\sim\mathrm{Gamma}(d,1)$ and $Y\sim\mathrm{Beta}(a,b)$ then for any $k\geq0$,
\begin{equation}\label{eq:mombetagamma}
\E[X^k] = \prod_{j=0}^{k-1}(d+j) \quad\text{and}\quad \E[Y^k]= \prod_{j=0}^{k-1}\frac{a+j}{a+b+j}.
\end{equation}

\noi In particular, we get
\begin{equation*}
\E[\xi^k]= k! \quad\text{and}\quad \E[\beta^k] = \frac{6}{(k+2)(k+3)}.
\end{equation*}

\noi Setting $q_k=\E[\beta^k]$ and using subsequently \eqref{eq:equationfonc}, the Binomial Theorem and the independence between $\beta$, $\xi$ and $T$, we obtain
\begin{align*}
t_k = \frac1{c^k}\frac1{k!}\E\big[(\beta^3(T+c\xi))^k\big]
 = \E\big[\beta^{3k}\big]\sum_{i=0}^k \frac1{c^k}\frac1{k!} \binom{k}{i} \E\big[T^i\big] c^{k-i}\E\big[\xi^{k-i}\big]
 = q_{3k}\sum_{i=0}^k t_i .
\end{align*}

\noi By isolating $t_k$, we deduce that
\begin{equation*}
t_k = \bigg(\frac1{1-q_{3k}}\bigg)\frac{q_{3k}}{q_{3(k-1)}}t_{k-1}.
\end{equation*}

\noi A telescopic product yields
\begin{equation*}
t_k = q_{3k}\bigg(\prod_{i=1}^k \frac1{1-q_{3i}}\bigg).
\end{equation*}

\noi Observing that 
$1-q_k= \mfrac{k(k+5)}{(k+2)(k+3)}$,
we get
\begin{align*}
\prod_{i=1}^k \frac1{1-q_{3i}} 
= \prod_{i=1}^k \frac{\left(\mfrac{2}{3}+i\right)\left(1+i\right)}{i\left(\mfrac{5}{3}+i\right)} = \frac1{k!}\bigg(\prod_{j=0}^{k-1}\left(\mfrac{5}{3}+j\right)\bigg)\bigg(\prod_{j=0}^{k-1}\frac{(2+j)}{\left(\mfrac{8}{3}+j\right)}\bigg).
\end{align*}

\noi It follows that, for any $k\geq0$,
\begin{equation}\label{eq:momT}
\E(T^k) = c^k k! t_k = c^k q_{3k}\bigg(\prod_{j=0}^{k-1}\left(\mfrac{5}{3}+j\right)\bigg)\bigg(\prod_{j=0}^{k-1}\frac{\left(2+j\right)}{\left(2+\mfrac{2}{3}+j\right)}\bigg).
\end{equation}

\noi Hence using \eqref{eq:mombetagamma} again, we can identify the product of the $k$-th moments of three random variables as follows:
\begin{equation*}
\E[T^k]= c^k\E[T_1^{3k}]\E[T_2^k]\E[T_3^k]
\end{equation*}
with $T_1\sim \mathrm{Beta}(2,2)$, $T_2\sim\mathrm{Gamma}\big(\tfrac{5}{3},1\big)$ and $T_3\sim\mathrm{Beta}\big(2,\tfrac{2}{3}\big)$.

\pass Upperbounding the $2k$-th moment of $T_1^3T_2T_3$ by $\E[T_2^{2k}]$ which is itself bounded by $(2k+1)!$, we verify that Carleman's condition is satisfied. This implies that the moments calculated at \eqref{eq:momT} characterize the distribution of $T$, i.e. $T\overset{(d)}{=}\tfrac32 T_1^3T_2T_3$.
\end{proof}

\subsection{Coupling of the Markov chains}\label{sec:couplage}

\pass 

\noi We prepare below the coupling of the Markov chains $\{(B_n^{(\la)},H_n^{(\la)})\}_n$ and $\{(B_n^{(\infty)},H_n^{(\infty)})\}_n$. First, we identify a good set $E^{(\la)}(\eps_1)$ for the starting position of the two Markov chains and also a good set $E^{(\la)}_{(b,h)}(\eps_1,\eps_2)$ which is both stable, see \eqref{eq:stabilityEbh}, and suitable for the coupling of the two Markov chains, see \eqref{eq:resultatamontrer}.

\pass More precisely, we introduce the set
\begin{equation}\label{eq:defEeps}
E^{(\la)}(\eps_1) = \big\{(b,h)\in (0,\infty)^2: b^3\le \la^{\eps_1}h \text{ and } b\le \la^{\eps_1}\big\}
\end{equation}
where the exponent $\varepsilon_1>0$ will be determined later.

\pass Similarly, for $(b,h)\in (0,\infty)^2$, we consider the set
\begin{equation}\label{eq:defEeps1ep2}
E^{(\la)}_{(b,h)}(\eps_1,\eps_2) =\Big\{(b',h')\in (0,\infty)^2 : b'^3\le \la^{\eps_1}h' \text{ and } b'\ge \la^{-\eps_2}b \Big\}
\end{equation}
where the exponents $\eps_1,\eps_2>0$ will be determined later.

\begin{prop}\label{prop3}
For every $\eps_0>0$ and $\eps_1,\eps_2>0$ small enough with respect to $\eps_0$, 
we obtain uniformly for every $b\ge \la^{\frac{\eps_0-1}{3}}$ and $(b,h)\in E^{(\la)}(\eps_1)$ the two following limits :
\begin{align}
& (i) \quad \lim_{\la\to\infty}(\log\la) \int_{\R^2\setminus E^{(\la)}_{(b,h)}(\eps_1,\eps_2)}p^{(\infty)}((b,h),(b',h'))\dd b'\dd h'=0.  \label{eq:stabilityEbh}\\ 
& (ii) \quad \lim_{\la\to\infty}(\log\la) \int_{E_{(b,h)}^{(\la)}(\eps_1,\eps_2)}|p^{(\la)}((b,h),(b',h'))-p^{(\infty)}((b,h),(b',h'))|\dd b'\dd h'=0. \label{eq:resultatamontrer}
\end{align}
\end{prop}

\begin{proof}
\noi $(i)$ Thanks to \eqref{eq:defEeps1ep2} and the fact that the support of $p^{(\infty)}((b,h),\cdot)$ is $(0,b)\times(0,h)$, we obtain that
\begin{align}\label{eq:intrecoupee}
\hspace*{-.5cm}\int_{\R^2\setminus E^{(\la)}_{(b,h)}(\eps_1,\eps_2)}p^{(\infty)}((b,h),(b',h'))\dd b'\dd h'&\le I_1+I_2
\end{align}
where
\begin{equation*}
I_1=\int_0^b\hspace*{-.2cm}\int_0^{h\wedge(\la^{-\eps_1}b'^3)}
\hspace*{-.1cm}p^{(\infty)}((b,h),(b',h'))\dd b'\dd h'\hspace*{.3cm}\text{and}\hspace*{.3cm} I_2=\int_0^{\la^{-\eps_2}b}\hspace*{-.2cm}\int_0^h
p^{(\infty)}((b,h),(b',h'))\dd b'\dd h'.
\end{equation*}

\noi Let us bound $I_1$. To do so, we take $\eta>0$ which will be fixed later and write
\begin{equation}\label{eq:intla4decomp}
I_1 \le J_1+J_2
\end{equation}
where
\begin{equation*}\label{eq:intla4decompJ1}
J_1 =\int_0^{(1-\la^{-\eta})b}\int_0^{\la^{-\eps_1}b'^3} p^{(\infty)}((b,h),(b',h'))\dd b'\dd h'
\end{equation*}
and
\begin{equation*}\label{eq:intla4decompJ2}
J_2= \int_{(1-\la^{-\eta})b}^b \int_0^h  p^{(\infty)}((b,h),(b',h'))\dd b'\dd h'.
\end{equation*}

\noi Let us deal with $J_1$ first. For any $(b,h)$ such that $\frac{b^3}{h}\le \la^{\eps_1}$ we get
\begin{align}
J_1
& = \int_0^{(1-\la^{-\eta})b} 4(b-b')b'\bigg(\int_0^{\la^{-\eps_1}b'^3}
\frac1{h'^2}\exp\bigg(-\frac23 b^3\bigg(\frac1{h'}-\frac1{h}\bigg)\bigg)\dd h'\bigg) \dd b' \nonumber\\
& = \int_0^{(1-\la^{-\eta})b}\frac{6(b-b')b'}{b^3} \bigg[\exp\bigg(-\frac23b^3\bigg(\frac1{h'}-\frac1{h}\bigg)\bigg)\bigg]_0^{\la^{-\eps_1}b'^3}\dd b' \nonumber\\
& = \exp\bigg(\frac23\frac{b^3}{h}\bigg)\int_0^{(1-\la^{-\eta})b}\frac{6(b-b')b'}{b^3}\exp\bigg(-\frac23\la^{\eps_1}\frac{b^3}{b'^3}\bigg)\dd b' \nonumber\\
& \le \exp\bigg(\frac23\la^{\eps_1} (1-(1-\la^{-\eta})^{-3})\bigg)\int_0^b\frac{6(b-b')b'}{b^3}\dd b' \nonumber\\
& \underset{\la\to\infty}{=}\exp\big(-2\la^{\eps_1-\eta}+\mathrm{o}(\la^{\eps_1-\eta})\big).\label{eq:herediteeps4}
\end{align}

\noi Therefore, choosing $\eta\in (0,\eps_1)$ implies that the quantity above goes to zero faster than any negative power of $\log\la$ when $\la\longrightarrow\infty$.

\pass We now turn our attention to $J_2$. For any $b,h>0$, we obtain
\begin{align}
J_2
& = \int_{(1-\la^{-\eta})b}^b \frac{6(b-b')b'}{b^3}\bigg[\exp\bigg(-\frac23b^3\bigg(\frac1{h'}-\frac1{h}\bigg)\bigg)\bigg]_0^h \dd b'\nonumber\\
& = \int_{(1-\la^{-\eta})b}^b \frac{6(b-b')b'}{b^3} \dd b'\nonumber\\
& = 1-3(1-\la^{-\eta})^2 + 2\la^{-\eta}(1-\la^{-\eta})^3 \nonumber\\
& = 3\la^{-2\eta}+\mathrm{o}(\la^{-2\eta}).\label{eq:herediteeps1}
\end{align}

\noi Again, the quantity above is negligible in front of any negative power of $\log\la$. Consequently, combining \eqref{eq:herediteeps4}, \eqref{eq:herediteeps1} and \eqref{eq:intla4decomp}, we obtain for every $(b,h)$ such that $b^3\le \la^{\eps_1}h$,
\begin{equation}\label{eq:condla4proved}
\lim_{\la\to\infty}(\log\la) I_1 = 0.
\end{equation}

\noi We now consider $I_2$. We proceed in the same way as in \eqref{eq:herediteeps1} and obtain
\begin{align*}
I_2
& = \int_0^{\la^{-\eps_2}b}\frac{6(b-b')b'}{b^3} \dd b'\le 3\la^{-2\eps_2}.\label{eq:herediteeps2}
\end{align*}

\noi This implies that
\begin{equation}\label{eq:condla2proved}
\lim_{\la\to\infty}(\log\la) I_2 = 0.
\end{equation}

\noi We complete the proof of \eqref{eq:stabilityEbh} by combining \eqref{eq:condla4proved}, \eqref{eq:condla2proved} and \eqref{eq:intrecoupee}.

\pass $(ii)$ Let us recall the definition of the set $S_{(b.h)}$ given at \eqref{eq:domainesautorisesbh}. We first prove that for $\eps_0$, $\eps_1$ and $\eps_2$ suitably chosen, there exists $\delta>0$ which is a positive linear combination of $\eps_0$, $\eps_1$ and $\eps_2$ such that the two following inequalities are satisfied: for every $b\ge \la^{\frac{\eps_0-1}{3}}$, $(b,h)\in E^{(\la)}(\eps_1)$ and $(b',h')\in E^{(\la)}_{(b,h)}(\eps_1,\eps_2)\cap [S_{(b,h)}\setminus (0,b)\times(0,h)]$,
\begin{equation}\label{eq:leldorado0}
p^{(\la)}((b,h),(b',h')) \le c \la^{\frac43}\la^{-\delta},
\end{equation}
and for every $b\ge \la^{\frac{\varepsilon_0-1}{3}}$, $(b,h)\in E^{(\la)}(\eps_1)$  and $(b',h')\in E^{(\la)}_{(b,h)}(\eps_1,\eps_2)\cap (0,b)\times(0,h)$,
\begin{equation}\label{eq:leldorado}
|p^{(\la)}((b,h),(b',h'))-p^{(\infty)}((b,h),(b',h'))|\le c \big(\mfrac{b'}{h'}\big)^2\la^{-\delta}
\end{equation}
where here and elsewhere, $c$ denotes a generic positive constant whose value may change at each occurence.

\noi In order to prove \eqref{eq:leldorado0} and \eqref{eq:leldorado}, we start by labelling each of the satisfied inequalities in the following way:
\begin{equation*}
\mbox{(C0)}\;\; b\ge \la^{\frac{\eps_0-1}{3}}\,\,;\,\,\mbox{(C1)}\;\; b^3\le \la^{\eps_1}h\;\mbox{ and }\; b'^3\le \la^{\eps_1}h'\,\,;\,\,\mbox{(C2)}\;\; b'\ge \la^{-\eps_2}b.
\end{equation*}

\noi We start with the case when $(b',h')\in S_{(b,h)}\setminus (0,b)\times(0,h)$. 
\noi In particular, $p^{(\infty)}((b,h),(b',h'))=0$ and in view of \eqref{eq:formuleexactpla},
\begin{align}\label{eq:case1h'grand}
0\le p^{(\la)}((b,h),(b',h')) & \le \frac{4bb'}{h'^2}\bigg(1-\la^{-\frac43}\frac{b^2}{h'^2}\bigg)^{-\frac12}.
\end{align}

\noi Using consecutively (C2), (C1) and (C0), we get
\begin{equation}\label{eq:1ereutilisationCi}
\la^{-\frac43}\frac{b^2}{h'^2}\le \la^{-\frac43+6\eps_2}\frac{b'^6}{b^4h'^2}\le \la^{-\frac43+2\eps_1+6\eps_2}\frac{1}{b^4}\le \la^{-(\frac43 \eps_0-2\eps_1-6\eps_2)}.
\end{equation}

\noi In particular, when $\eps_0>\frac32 \eps_1+\frac92\eps_2$, we obtain that for $\la$ large enough,
\begin{equation}\label{eq:terme1cas1}
0\le \bigg(1-\la^{-\frac43}\frac{b^2}{h'^2}\bigg)^{-\frac12}\le c.
\end{equation}

\noi Moreover, with the same consecutive use of (C2), (C1) and (C0), we get
\begin{equation}\label{eq:terme0cas1}
\frac{bb'}{h'^2}\le \la^{5\eps_2}\frac{b'^6}{b^4h'^2}\le \la^{\frac43}\la^{-(\frac43 \eps_0-2\eps_1-5\eps_2)}.
\end{equation}

\noi Inserting \eqref{eq:terme1cas1} and \eqref{eq:terme0cas1} into \eqref{eq:case1h'grand}, we obtain \eqref{eq:leldorado0} with the provisional choice $\delta=\frac43 \eps_0-2\eps_1-5\eps_2>0$ as soon as $\eps_0>\frac32 \eps_1+\frac{15}{4}\eps_2$.

\pass We now deal with the case when $(b', h')\in (0,b)\times(0,h).$ We rewrite $p^{(\la)}$ given at \eqref{eq:formuleexactpla} (resp. $p^{(\infty)}$ given at \eqref{eq:transideale}) as $q^{(\la)}e^{-a^{(\la)}}$ (resp. $q^{(\infty)}e^{-a^{(\infty)}}$) where $a^{(\la)}=|\Delta_n^{(\la)}|_2$ (resp. $a^{(\infty)}=\lim_{\la\to\infty}|\Delta_n^{(\la)}|_2$). We omit the arguments of each of these functions for sake of simplicity. We then obtain
\begin{align}
|p^{(\la)}-p^{(\infty)}| 
& \le e^{-a^{(\la)}}|q^{(\la)}-q^{(\infty)}|+q^{(\infty)}|e^{-a^{(\la)}}-e^{-a^{(\infty)}}| \nonumber \\
& \le |q^{(\la)}-q^{(\infty)}|+\la^{\frac43}\la^{-(\frac43 \eps_0-2\eps_1-5\eps_2)}|a^{(\la)}-a^{(\infty)}|  \label{eq:cas21ereineg}
\end{align}
where we have used \eqref{eq:terme0cas1} to bound $q^{(\infty)}$.

\pass We start by bounding the first term in the right-hand side of \eqref{eq:cas21ereineg}.
\begin{align}\label{eq:cas21ertermerhs}
|q^{(\la)}-q^{(\infty)}|
& = \frac{4b'}{h'^2}\bigg|b\big(1-\big(1+\la^{-\frac43}\mfrac{b'^2}{h'^2})^{-1}\big(1-\la^{-\frac43}\mfrac{b^2-b'^2}{h'^2}\big)^{-\frac12}\big)
-b'\big(1-\big(1+\la^{-\frac43}\mfrac{b'^2}{h'^2}\big)^{-1}\big)\bigg| \nonumber \\
& \le\frac{4b'}{h'^2}\left(b\frac{\big|(1+\la^{-\frac43}\frac{b'^2}{h'^2})
(1-\la^{-\frac43}\frac{b^2-b'^2}{h'^2})^{\frac12}-1\big|}{(1+\la^{-\frac43}\frac{b'^2}{h'^2})(1-\la^{-\frac43}\frac{b^2-b'^2}{h'^2})^{\frac12}}
+ b'\frac{\la^{-\frac43}\frac{b'^2}{h'^2}}{1+\la^{-\frac43}\frac{b'^2}{h'^2}}\right).
\end{align}

\noi Thanks to \eqref{eq:1ereutilisationCi}, we obtain that
\begin{equation*}
\Big|(1+\la^{-\frac43}\mfrac{b'^2}{h'^2})(1-\la^{-\frac43}\mfrac{b^2-b'^2}{h'^2})^{\frac12}-1\Big|\le c\la^{-\frac43}\mfrac{b^2}{h'^2}\le c\la^{-(\frac43 \eps_0-2\eps_1-6\eps_2)}.
\end{equation*}

\noi Consequently, we deduce that
\begin{equation}\label{eq:1erterme1ertermerhscas2}
b\frac{\big|(1+\la^{-\frac43}\frac{b'^2}{h'^2})(1-\la^{-\frac43}\frac{b^2-b'^2}{h'^2})^{\frac12}-1\big|}{(1+\la^{-\frac43}\frac{b'^2}{h'^2})(1-\la^{-\frac43}\frac{b^2-b'^2}{h'^2})^{\frac12}}\le cb\la^{-(\frac43 \eps_0-2\eps_1-6\eps_2)}\le c b'\la^{-(\frac43 \eps_0-2\eps_1-7\eps_2)}.
\end{equation}

\noi In the same way,
\begin{equation}\label{eq:2emeterme1ertermerhscas2}
b'\frac{\la^{-\frac43}\frac{b'^2}{h'^2}}{1+\la^{-\frac43}\frac{b'^2}{h'^2}}\le cb'\la^{-(\frac43 \eps_0-2\eps_1-7\eps_2)}.
\end{equation}

\noi Inserting \eqref{eq:1erterme1ertermerhscas2} and \eqref{eq:2emeterme1ertermerhscas2} into \eqref{eq:cas21ertermerhs}, we deduce that
\begin{equation}\label{eq:1ertermecas2regle}
|q^{(\la)}-q^{(\infty)}|\le c\bigg(\frac{b'^2}{h'^2}\bigg)\la^{-(\frac43 \eps_0-2\eps_1-7\eps_2)}.
\end{equation}
                    
\noi We now consider the second term in the right-hand side of \eqref{eq:cas21ereineg}, i.e. $|a^{(\la)}-a^{(\infty)}|$. We start by rewriting $a^{(\la)}=|\Delta_n^{(\la)}|_2$ with the intermediary function $\xi(x)=\arctan(x)-x$:

\begin{align}\label{eq:rewritinga}
a^{(\la)} 
= \la^2 h'^2 & \big(1+\la^{-\frac43}\mfrac{b'^2}{h'^2}\big)\xi\bigg(\frac{\la^{-\frac23}\frac{b}{h'}}{(1-\la^{-\frac43}\frac{b^2-b'^2}{h'^2})^{\frac12}}\bigg) + \frac{\frac{b^3}{h'}}{(1-\la^{-\frac43}\frac{b^2-b'^2}{h'^2})^{\frac12}} \nonumber \\ 
& -\la^2 h^2\big(1+\la^{-\frac43}\mfrac{b^2}{h^2}\big)\xi\big(\mfrac{\la^{-\frac23}b}{h}\big)-\mfrac{b^3}{h}.
\end{align}

\noi Thanks to \eqref{eq:1ereutilisationCi} and the inequality $(1-x)^{-\frac12}\le 1+x$ for $x>0$ small enough, we notice that
\begin{equation}\label{eq:argumentdearctan}
0\le \frac{\la^{-\frac23}\frac{b}{h'}}{(1-\la^{-\frac43}\frac{b^2-b'^2}{h'^2})^{\frac12}}-\la^{-\frac23}\big(\mfrac{b}{h'}\big) \le \la^{-2}\big(\mfrac{b^3}{h'^3}\big).
\end{equation}

\noi Moreover, we recall that for any $x>0$,
\begin{equation}\label{eq:inegdearctan}
0\le \xi(x)+\frac{x^3}{3}\le \frac{x^5}{5}
\end{equation}

\noi Inserting both \eqref{eq:argumentdearctan} and \eqref{eq:inegdearctan} into \eqref{eq:rewritinga} and remembering that $a^{(\infty)}$ is given at \eqref{eq:limdeltan}, we deduce that 
\begin{align*}
a^{(\la)}-a^{(\infty)} 
& \le -\la^2 h'^2 \big(1+\la^{-\frac43}\mfrac{b'^2}{h'^2}\big)\la^{-2}\big(\mfrac{b^3}{3h'^3}\big)+\mfrac{\la^2h'^2}{5}\big(1+\la^{-\frac43}\mfrac{b'^2}{h'^2}\big)\big(\la^{-\frac23}\mfrac{b}{h'}
+ \la^{-2}\mfrac{b^3}{h'^3}\big)^5 \\ 
& \hspace*{2cm} +\mfrac{b^3}{h'}\big(1+\la^{-\frac43}\mfrac{b^2}{h'^2}\big)
+ \la^2 h^2 \big(1+\la^{-\frac43}\mfrac{b^2}{h^2}\big)\la^{-2}\mfrac{b^3}{3h^3} -\mfrac{b^3}{h} - \mfrac23 b^3\big(\mfrac1{h'}-\mfrac1{h}\big) \\
& \le \la^{-\frac43}\mfrac{b^5}{5h'^3}\big(1+\la^{-\frac43}\mfrac{b^2}{h'^2}\big)^6+\la^{-\frac43}\mfrac{b^5}{h'^3}+\la^{-\frac43}\mfrac{b^5}{3h^3} \\
& \le c\la^{-\frac43}\mfrac{b^5}{h'^3} \\
& \le c \la^{-\frac43}\la^{\eps_1+5\eps_2}\mfrac{b'^2}{h'^2}
\end{align*}
where we have obtained the last line by using (C2) then (C1).

\pass Similarly, we verify that $(a^{(\infty)}-a^{(\la)})$ has the same upper-bound and this shows that
\begin{equation}\label{eq:alambdatermine}
\la^{\frac43}|a^{(\la)}-a^{(\infty)}|\le \la^{\eps_1+5\eps_2}\mfrac{b'^2}{h'^2}.
\end{equation}

\noi Inserting \eqref{eq:1ertermecas2regle} and \eqref{eq:alambdatermine} into \eqref{eq:cas21ereineg}, we finally obtain \eqref{eq:leldorado} with the choice
$\delta=\frac43 \eps_0-3\eps_1-10\eps_2>0$ as soon as $\eps_0>\frac94 \eps_1+\frac{15}{2}\eps_2$. Naturally, this updated choice of $\delta$ is consistent with \eqref{eq:leldorado0} as well. 

\pass We are now ready to show \eqref{eq:resultatamontrer}. Noticing that
\begin{align}\label{eq:decompintplapinfty}
& \int_{E_{(b,h)}^{(\la)}(\eps_1,\eps_2)} |p^{(\la)}((b,h),(b',h'))-p^{(\infty)}((b,h),(b',h'))|\dd b'\dd h' \nonumber \\
& \hspace{0.2cm}\le\int_{E_{(b,h)}^{(\la)}(\eps_1,\eps_2)\cap (0,b)\times(0,h)}|p^{(\la)}((b,h),(b',h'))-p^{(\infty)}((b,h),(b',h'))| \dd b'\dd h'\nonumber \\ 
& \hspace*{2cm} +\int_0^b\int_h^{(h^2+\la^{-\frac43}(b^2-b'^2))^{\frac12}}p^{(\la)}((b,h),(b',h')) \dd b'\dd h',
\end{align}
we observe that the two following estimates are enough to derive \eqref{eq:resultatamontrer}:

\begin{align}
& \lim_{\la\to\infty}(\log\la) \int_{E_{(b,h)}^{(\la)}(\eps_1,\eps_2)\cap(0,b)\times(0,h)}|p^{(\la)}((b,h),(b',h'))-p^{(\infty)}((b,h),(b',h'))|
\dd b'\dd h'= 0, \label{eq:1ereintplapinfty} \\
& \lim_{\la\to\infty}(\log\la) \int_0^b\int_h^{(h^2+\la^{-\frac43}(b^2-b'^2))^{\frac12}} p^{(\la)}((b,h),(b',h'))
\dd b'\dd h' = 0 .\label{eq:2emeintplapinfty}
\end{align}

\noi We start by proving \eqref{eq:2emeintplapinfty}. Inserting the right-hand side of \eqref{eq:leldorado0} into the integral and using the inequality $(1+x)^{\frac12}\le 1+\frac{x}{2}$ for $x>0$, we get for any $(b,h)\in E^{(\la)}(\eps_1)$ such that $b\ge \la^{\frac{\eps_0-1}{3}}$:
\begin{align*}
\int_0^b\int_h^{(h^2+\la^{-\frac43}(b^2-b'^2))^{\frac12}}p^{(\la)}((b,h),(b'h'))\dd b'\dd h' 
& \le c\la^{\frac43-\delta}\int_0^bh((1+\la^{-\frac43}\frac{b^2-b'^2}{h^2})^{\frac12}-1)\dd b' \\
& \le c\la^{-\delta}\frac{1}{h}\int_0^b(b^2-b'^2)\dd b'\\
& = \mfrac{2c}{3}\la^{-\delta}\mfrac{b^3}{h} \\
& \le c\la^{-\delta+\eps_1}.
\end{align*}

\noi We conclude that the integral is negligible in front of any negative power of $\log\la$ as soon as $\delta-\eps_1>0$, which occurs when $\eps_0>3\eps_1+\frac{15}{2}\eps_2.$ This proves \eqref{eq:2emeintplapinfty}.

\pass We now turn our attention to \eqref{eq:1ereintplapinfty}. Integrating the right-hand side from \eqref{eq:leldorado} for $h'\in (\la^{-\eps_1}b'^3,h)$ (condition (C1)) and $b'\in (\la^{\frac{\eps_0-1}{3}-\eps_2},b)$ (conditions (C0) and (C2)), we obtain
\begin{align*}
\int_{\la^{\frac{\eps_0-1}{3}-\eps_2}}^b
 & \int_{\la^{-\eps_1}b'^3}^h\big|p^{(\la)}((b,h),(b',h'))-p^{(\infty)}((b,h),(b',h'))\big|\dd b'\dd h' \\
 & \le c\la^{-\delta}\int_{\la^{\frac{\eps_0-1}{3}-\eps_2}}^b b'^2\left[-\frac1{h'}\right]_{\la^{-\eps_1}b'^3}^h\dd b'\\
 & \le c\la^{-\delta+\eps_1}\int_{\la^{\frac{\eps_0-1}{3}-\eps_2}}^b \frac1{b'}\dd b' \\
 & \le c\la^{-\delta+\eps_1}\log \la,
\end{align*}
where the last inequality comes from the condition $b\le \la^{\eps_1}$ satisfied by any $(b,h)\in E^{(\la)}(\eps_1)$. 

\pass We conclude that the integral is negligible in front of any negative power of $\log\la$ as soon as $\delta-\eps_1=\frac43\eps_0-4\eps_1-10\eps_2>0$. This proves \eqref{eq:1ereintplapinfty} and in turn \eqref{eq:resultatamontrer} thanks to \eqref{eq:decompintplapinfty}. The proof of Proposition \ref{prop3} is complete.
\end{proof}

\noi From now on we assume that Markov chains $\{(B_n^{(\la)},H_n^{(\la)})\}_n$ and $\{(B_n^{(\infty)},H_n^{(\infty)})\}_n$ have both same initial distribution given at \eqref{eq:defBoHo}. We decide from now on to couple the two Markov chains $\{(B_n^{(\la)},H_n^{(\la)})\}_n$ and $\{(B_n^{(\infty)},H_n^{(\infty)})\}_n$ with respective transition probabilities $p^{(\la)}$ and $p^{(\infty)}$ in the following way. We define $(B_0^{(\la)},H_0^{(\la)})=(B_0^{(\infty)},H_0^{(\infty)})$ as in \eqref{eq:defBnHn} setting
\begin{equation*}
B_0^{(\la)}=B_0^{(\infty)}=\frac{\la^{-\frac13}}{2}\|Z_0-Z_{\textsl r}^{(\la)}\| \quad\text{and}\quad H_0^{(\la)}=H_0^{(\infty)}=\la^{-1}\dd(V_0^{(\la)},(Z_{\textsl r}^{(\la)},Z_0)),
\end{equation*}
and then we recursively construct the coupling as follows. 

\noi Let us assume that $(B_0^{(\la)},H_0^{(\la)}),\ldots,(B_n^{(\la)},H_n^{(\la)})$ and  $(B_0^{(\infty)},H_0^{(\infty)}),\ldots, (B_n^{(\infty)},H_n^{(\infty)})$ have been constructed for all $0\le k\le n$. We deal with the two following cases:
\begin{itemize}[parsep=-0.15cm,itemsep=0.25cm,topsep=0.2cm,wide=0.15cm,leftmargin=0.5cm]
\item[-] If $(B_k^{(\la)},H_k^{(\la)})= (B_k^{(\infty)},H_k^{(\infty)})$ for all $0\le k\le n$, we define the two couples $({B_{n+1}^{(\la)}},{H_{n+1}^{(\la)}})$ and $(B_{n+1}^{(\infty)},H_{n+1}^{(\infty)})$ so that they coincide with probability $$1-\mfrac12 \iint |p^{(\la)}((B_n^{(\la)},H_n^{(\la)}),(b',h'))-p^{(\infty)}((B_n^{(\la)},H_n^{(\la)}),(b',h'))|\indicat_{S_{(b,h)}}(b',h')\dd b'\dd h',$$
\item[-] If there exists $0\le k\le n$ such that $(B_k^{(\la)},H_k^{(\la)})\ne (B_k^{(\infty)},H_k^{(\infty)})$, we define $({B_{n+1}^{(\la)}},{H_{n+1}^{(\la)}})$ and $(B_{n+1}^{(\infty)},H_{n+1}^{(\infty)})$ independently with respective distributions $p^{(\la)}((B_n^{(\la)},H_n^{(\la)}),\cdot)$ and $p^{(\infty)}((B_n^{(\infty)},H_n^{(\infty)}),\cdot)$.
\end{itemize}

\pass We will show that with high probability, the coupling is exact until a certain stopping time that we now describe. For $\eps_0>0$ we denote by $\tau_\la(\eps_0)$ the stopping time
\begin{equation}\label{eq:deftaula}
\tau_\la(\eps_0)=\inf\Big\{n\ge 1 : B_n^{(\infty)}<\la^{\frac{\eps_0-1}{3}}\Big\}.
\end{equation}
In Lemma \ref{lem:taulainf}, we deduce from Corollary \ref{cor:logBn} a weak concentration result for $\tau_\la(\eps_0)$.
\begin{lem}\label{lem:taulainf}
For every $\eps_0\in (0,1)$ and $\eta>\frac25 \eps_0$,
\begin{equation}\label{eq:grandedevtaula}
\lim_{\la\to\infty}\P\Big[|\tau_\la(\eps_0)-\mfrac25 \log\la|> \eta\log\la\Big]=0.
\end{equation}
\end{lem}

\begin{proof}
\noi Thanks to Corollary \ref{cor:logBn}, we obtain for $\eta>0$
\begin{align}\label{eq:ldp}
\P\Big[\tau_\la(\eps_0)> (\mfrac25+\eta)\log\la\Big]
= \P\bigg[B_{(\frac25+\eta)\log\la}^{(\infty)}>\la^{\frac{\eps_0-1}{3}}\bigg] = \P\bigg[L_{(\frac25+\eta)\log\la}< \frac{1-\eps_0}{\frac65+3\eta}\bigg].
\end{align}

\noi Consequently, the right-hand side in \eqref{eq:ldp} is a classical large deviation probability which goes to 0 when $\la\longrightarrow\infty$. Moreover, thanks to the monotonicity of the sequence  $\{B_n^{(\infty)}\}_n$, for $\eta\in (0,\frac25)$,
\begin{align*}
\P\Big[\tau_\la(\eps_0)<(\mfrac25-\eta)\log\la \Big] =\P\bigg[B_{(\frac25-\eta)\log\la}^{(\infty)} <\la^{\frac{\eps_0-1}{3}}\bigg]
=\P\bigg[L_{(\frac25-\eta)\log\la}>\frac{1-\eps_0}{\frac65-3\eta}\bigg].
\end{align*}
When $\eta>\frac25 \eps_0$, the right-hand side in \eqref{eq:ldp} is a classical large deviation probability which goes to 0 when $\la\longrightarrow\infty$.
\end{proof}

\pass Let us consider the event $E_{\mbox{\tiny{coupl}}}^{(\la)}(\eps_0)$ which garantees the perfect coupling until time $\tau_\la(\eps_0)$, i.e.

\begin{equation}\label{eq:defpascouplage}
E_{\mbox{\tiny{coupl}}}^{(\la)}(\eps_0) = \bigcap_{n=1}^{\tau_\la(\eps_0)}\Big\{ (B_n^{(\la)},H_n^{(\la)})= (B_n^{(\infty)},H_n^{(\infty)}) \Big\}.
\end{equation}

\noi The main result of this Section is the following.

\begin{prop}\label{prop:evenpascouplage}
For every $\eps_0>0$,
\begin{equation*}
\lim_{\la\to\infty}\P\big[E_{\emph{\tiny{coupl}}}^{(\la)}(\eps_0)\big]=1.
\end{equation*}
\end{prop}

\noi Proving Proposition \ref{prop:evenpascouplage} requires to introduce beforehand a particular event with high probability where the couples $(B_n^{(\infty)},H_n^{(\infty)})$ for $n\le \tau^{(\infty)}_{\la}$ behave nicely, see \eqref{eq:Egood}. This takes place in Lemma \ref{lem:goodevent} below. Then we prove Proposition \ref{prop:evenpascouplage} at the end of Section \ref{sec:couplage}.

\pass For $\eps_0,\eps_1,\eps_2>0$, we introduce
\begin{align}
E_{\mbox{\tiny{good}}}^{(\la)}=E_{\mbox{\tiny{good}}}^{(\la)}(\eps_0,\eps_1,\eps_2) & =
\bigcap_{n=1}^{\tau^{(\infty)}_\la(\eps_0)} \Big\{(B_n^{(\infty)},H_n^{(\infty)})\in E^{(\la)}_{(B_{n-1}^{(\infty)},H_{n-1}^{(\infty)})}(\eps_1,\eps_2)\Big\}\nonumber \\
& \hspace*{1.2cm}\bigcap\Big\{(B_0^{(\infty)},H_0^{(\infty)})\in E^{(\la)}(\eps_1)\Big\} \bigcap \Big\{\tau_\la(\eps_0)\le \log\la\Big\}. \label{eq:Egood}
\end{align}

\noi Whenever possible, we will omit the dependency of $E_{\mbox{\tiny{good}}}^{(\la)}$ with respect to $\eps_0,\eps_1,\eps_2$ for sake of readability.

\begin{lem}\label{lem:goodevent}
For every $\eps_0,\eps_1,\eps_2$ such that $\eps_0>4\eps_1+\frac{15}{2}\eps_2$,
\begin{equation*}
\lim_{\la\to\infty}\P\Big[E_{\emph{\tiny{good}}}^{(\la)}(\eps_0,\eps_1,\eps_2)\Big]=1.
\end{equation*}
\end{lem}

\begin{proof} Thanks to the convergence in distribution of $(B_0^{(\la)},H_0^{(\la)})$ stated in Lemma \ref{lem:convBoHo}, we get for every $\eps_1>0$
\begin{equation}\label{eq:BoHopasdansEla}
\lim_{\la\to\infty}\P\Big[(B_0^{(\la)},H_0^{(\la)})\in E^{(\la)}(\eps_1)\Big]=1.
\end{equation}

\noi We can now upper-bound the probability of $(E^{(\la)}_{\mbox{\tiny{good}}})^c$ as follows:
\begin{align}\label{eq:decompprobaevencontr}
\P\big[(E^{(\la)}_{\mbox{\tiny{good}}})^c\big]
& \le \P[\tau_\la>\log\la]+\P\big[(B_0^{(\infty)},H_0^{(\infty)})\not\in E^{(\la)}(\eps_1)\big]\nonumber \\
& \hspace{-1cm} + \sum_{n=1}^{\log\la}\P\bigg[\big\{(B_{n-1}^{(\infty)},H_{n-1}^{(\infty)})\in E^{(\la)}(\eps_1)\big\}\cap \big\{(B_n^{(\infty)},H_n^{(\infty)})\not\in E^{(\la)}_{(B_{n-1}^{(\infty)},H_{n-1}^{(\infty)})}(\eps_1,\eps_2)\big\}\bigg]
\end{align}
where we have used the inclusion $E^{(\la)}_{(b,h)}(\eps_1,\eps_2)\subset E^{(\la)}(\eps_1).$

\pass Moreover,
\begin{align*}
\P\bigg[\big\{(B_{n-1}^{(\infty)},H_{n-1}^{(\infty)}) 
& \in E^{(\la)}(\eps_1) \big\}\cap \big\{(B_n^{(\infty)},H_n^{(\infty)})\not\in E^{(\la)}_{(B_{n-1}^{(\infty)},H_{n-1}^{(\infty)})}(\eps_1,\eps_2)\big\}\bigg] \\
& \le \int_{E^{(\la)}(\eps_1)}\bigg(\int_{\R^2\setminus E^{(\la)}_{(b,h)}(\eps_1,\eps_2)} p^{(\infty)}((b,h),(b',h'))\dd b' \dd h'\bigg)
\dd \P_{(B_{n-1}^{(\infty)},H_{n-1}^{(\infty)})}(b,h)
\end{align*}
where $\P_{(B_{n-1}^{(\infty)},H_{n-1}^{(\infty)})}$ denotes the distribution of ${(B_{n-1}^{(\infty)},H_{n-1}^{(\infty)})}$. Thanks to Proposition \ref{prop3} (i), the integral inside the brackets is $\mathrm{o}((\log\la)^{-1})$ uniformly in $(b,h)$, which implies that
\begin{equation}\label{eq:sumlogtermsneglig}
\lim_{\la\to\infty} \sum_{n=1}^{\log\la}\P\bigg[\big\{(B_{n-1}^{(\infty)},H_{n-1}^{(\infty)})\in E^{(\la)}(\eps_1)\big\}\cap\big\{(B_n^{(\infty)},H_n^{(\infty)})\not\in E^{(\la)}_{(B_{n-1}^{(\infty)},H_{n-1}^{(\infty)})}(\eps_1,\eps_2)\big\}\bigg]=0.
\end{equation}

\noi Consequently, inserting \eqref{eq:sumlogtermsneglig}, \eqref{eq:BoHopasdansEla} and \eqref{eq:grandedevtaula} into \eqref{eq:decompprobaevencontr}, we complete the proof of Lemma \ref{lem:goodevent}.
\end{proof}

\pass\textbf{Proof of Proposition \ref{prop:evenpascouplage}.}

\noi We now estimate the probability of $E_{\mbox{\tiny{coupl}}}^{(\la)}$. First, we notice that
\begin{align}\label{eq:decompPAla}
\P\big[(E_{\mbox{\tiny{coupl}}}^{(\la)})^c\big] \le 1-\P\big[E^{(\la)}_{\mbox{\tiny{good}}}\big]+\P\big[({E_{\mbox{\tiny{coupl}}}^{(\la)}})^c\cap E^{(\la)}_{\mbox{\tiny{good}}}\big].
\end{align}

\noi Thanks to Lemma \ref{lem:goodevent}, \eqref{eq:decompPAla} implies that it is enough to show that
\begin{equation}\label{eq:probapetite}
\lim_{\la\to\infty}\P\big[({E_{\mbox{\tiny{coupl}}}^{(\la)}})^c\cap E^{(\la)}_{\mbox{\tiny{good}}}\big]=0.
\end{equation}

\noi Below, we bound this probability in order to use \eqref{eq:resultatamontrer}.
\begin{align}
& \P\big[({E_{\mbox{\tiny{coupl}}}^{(\la)}})^c\cap E^{(\la)}_{\mbox{\tiny{good}}}\big] \\
& \le
\sum_{n=1}^{\log\la}\P\left[E^{(\la)}_{\mbox{\tiny{good}}}\,;\,\forall\;i\le n, (B_i^{(\la)},H_i^{(\la)})=(B_i^{(\infty)},H_i^{(\infty)}),\right.\nonumber\\
& \left.\hspace*{2.5cm}(B_{n+1}^{(\la)},H_{n+1}^{(\la)})\ne(B_{n+1}^{(\infty)},H_{n+1}^{(\infty)}),\tau_\la(\eps_0)\ge n+1\right]\nonumber \\
& \le \sum_{n=1}^{\log\la}\P\left[\forall\;i\le n, (B_i^{(\la)},H_i^{(\la)})=(B_i^{(\infty)},H_i^{(\infty)}),B_n^{(\infty)}\ge \la^{\frac{\varepsilon_0-1}{3}},(B_n^{(\infty)},H_n^{(\infty)})\in E^{(\la)}(\eps_1),\right.\nonumber\\
&\left.\hspace*{2.5cm} (B_{n+1}^{(\la)},H_{n+1}^{(\la)})\ne(B_{n+1}^{(\infty)},H_{n+1}^{(\infty)}),(B_{n+1}^{(\infty)},H_{n+1}^{(\infty)})\in E^{(\la)}_{(B_n,H_n)}(\eps_1,\eps_2)\right]\nonumber.
\end{align}

\noi Let us denote by $q_{n,\la} $ the probability which appears in the sum in the right-hand side above and let us show that uniformly in $1\le n\le \log\la$,
\begin{equation}\label{eq:qklaneglig}
\lim_{\la\to\infty}(\log\la) q_{n,\la}=0.
\end{equation}

\noi Indeed,
\begin{align*}
(\log\la) q_{n,\la} &\le \int_{\la^{\frac{\eps_0-1}{3}}}^{\la^{\eps_1}}\int_{\la^{-\eps_1}b^3}^{\infty} \bigg(\mfrac{\log\la}{2}\int_{E^{(\la)}_{(b,h)}(\eps_1,\eps_2)} |p^{(\la)}((b,h),(b',h'))\\&\hspace*{5cm}-p^{(\infty)}((b,h),(b',h'))|\dd b' \dd h'\bigg)
\dd \P_{(B_n^{(\infty)},H_n^{(\infty)})}(b,h).
\end{align*}

\noi Thanks to \eqref{eq:resultatamontrer}, the quantity inside the brackets goes to 0 uniformly in $(b,h)$ when $b\ge \la^{\frac{\eps_0-1}{3}}$ and $(b,h)\in E^{(\la)}(\eps_1)$. Consequently, we deduce \eqref{eq:qklaneglig} and the proof is complete. \hfill$\square$

\addtocontents{toc}{\vspace{0.25cm}}%
\section{Proofs of the main results}\label{sec:mainproof}

\pass This section is devoted to the proof of Theorem \ref{theo:menhir}. Our approach consists in dividing the boundary of $\cC^{(\la)}$ into two branches, left and right, which are independent conditional on $(Z_{\textsl l}^{(\la)},Z_{\textsl c}^{(\la)},Z_{\textsl r}^{(\la)})$. In particular, we recall that $\cB^{(\la)}$ denotes the left branch, i.e. the polygonal line which joins all consecutive vertices on the left-hand side of the Voronoi cell associated with $Z_{\textsl c}^{(\la)}$ from the point $(0,\la)$ to the $x$-axis. 
In Section \ref{sec:shapebranch}, we prove part (i) of Theorem \ref{theo:menhir}, i.e. the limit shape of the renormalized Voronoi cell, by showing the convergence of the shape of $\cB^{(\la)}$, see Proposition \ref{prop:estquecestpossible}. Sections \ref{sec:preuveTH2} and \ref{sec:preuveTH3} include the proofs of parts (ii) and (iii) of Theorem \ref{theo:menhir}, namely the local behavior at the origin and the asymptotics for the number of vertices of $\cC^{(\la)}$, which are consequences of their counterparts for $\cB^{(\la)}$, see Propositions \ref{prop:unbusrapide} and \ref{prop:cestpossible} respectively.

\subsection{Proof of Theorem \ref{theo:menhir} (i): convergence to the menhir}\label{sec:shapebranch}

\pass 

\noi Keeping in mind the decomposition of the boundary of $\cC^{(\la)}$ into two independent branches conditional on $(Z_{\textsl l}^{(\la)},Z_{\textsl c}^{(\la)},Z_{\textsl r}^{(\la)})$, we claim that the convergence in distribution of $F^{(\la)}(\cC^{(\la)}-Z_{\textsl c}^{(\la)})$ to the limiting menhir $\cC^{(\infty)}$ described in Section \ref{sec:intro} is a direct consequence of the convergence of the left branch  $\cB^{(\la)}$ stated in Proposition \ref{prop:estquecestpossible} below.

\begin{prop}\label{prop:estquecestpossible}
We endow the set of non empty compact sets of $\R^2$ with the topology of the Hausdorff metric. When $\la\longrightarrow\infty$, we obtain
\begin{equation*}
F^{(\la)}(\cB^{(\la)}-Z_{\textsl c}^{(\la)})\overset{(d)}{\longrightarrow} \cB_{\mathbf l}
\end{equation*}
where $\cB_{\mathbf l}$ is defined in part 3 on page \pageref{defbranchBl}.
\end{prop}

\begin{proof} 
The proof is essentially a consequence of the convergence of the sequence of vertices of the branch before and after the coupling time $\tau_\la(\eps_0)$. Indeed, Lemmas \ref{lem:aimerjusqualimpossible} and \ref{lem:lesortdespointsaprescouplage} show the existence of a realization of $\{V_n^{\mathbf l}\}_n$ such that there is uniform convergence in probability of the sequence of renormalized vertices after translation $\{F^{(\la)}(V_n^{(\la)}-Z_{\textsl c}^{(\la)})\}_n$ to the sequence $\{V_n^{\mathbf l}\}_n$, i.e.
\begin{equation*}
\max_{n\ge 0}\|F^{(\la)}(V_n^{(\la)}-Z_{\textsl c}^{(\la)})-V_n^{\mathbf l}\|\overset{\P}{\longrightarrow} 0.
\end{equation*}
Let us notice that $\cB_{\mathbf l}$ is almost surely a compact set. This is due to the fact that $V_n^{\mathbf l}\longrightarrow 0$ almost surely when $n\to \infty$. Hence, by continuity of the function $\conv: \cK\longrightarrow\cK$ (see e.g. \cite[Theorem 12.3.5]{SW08}) the convergence in distribution of 
$F^{(\la)}(\cB^{(\la)}-Z_{\textsl c}^{(\la)})$ 
to $\cB_{\mathbf l}$ is a consequence of the convergence in distribution of the random compact set $\{F^{(\la)}(V_n^{(\la)}-Z_{\textsl c}^{(\la)})\}_n$ to the closure of $\{V_n^{\mathbf l}\}_n$.
\end{proof}

\noi We recall that Lemma \ref{lem:lessommetsretrouves} provides exact formulas for the distance between two consecutive vertices from the sequence $\{V_n^{(\la)}\}_n$ and the angle between two consecutive pairs of such vertices as functions of $\{(B_n^{(\la)},H_n^{(\la)})\}_n$. In Lemma \ref{lem:lessommetsretrouves2} we use these formulas to deduce asymptotic estimates that pave the way for future Taylor expansions done in Lemma \ref{lem:aimerjusqualimpossible}.

\begin{lem}\label{lem:lessommetsretrouves2}
Let $\eps_0=\frac13+\eps_0'$ where $\eps_0'\in (0,\frac23)$ and let $\eps_1,\eps_2>0$ be small enough. On $E^{(\la)}_{\emph{\tiny{good}}}(\eps_0,\eps_1,\eps_2)\cap E^{(\la)}_{\emph{\tiny{coupl}}}(\eps_0)$, there exists $\delta>0$ such that for $\la$ large enough and $n\le \tau_\la(\eps_0)$,
\begin{align}\label{eq:normesapprox}
0\le \la^{-1}\|V_n^{(\la)}-V_{n+1}^{(\la)}\|-(H_n^{(\infty)}-H_{n+1}^{(\infty)})\le \la^{-\frac{10}{9}-\delta}
\end{align}
and
\begin{align}\label{eq:anglesapprox}
\bigg|\la^{\frac23}\angle (V_n^{(\la)}V_{n+1}^{(\la)}V_{n+2}^{(\la)})-\frac{B_n^{(\infty)}-{B_{n+1}^{(\infty)}}}{{H_{n+1}^{(\infty)}}}\bigg| \le \la^{-\delta}.
\end{align}
\end{lem}

\begin{proof}
On the event $E^{(\la)}_{\mbox{\tiny{good}}}(\eps_0,\eps_1,\eps_2)\cap  E^{(\la)}_{\mbox{\tiny{coupl}}}(\eps_0)$, the sequence $\{(B_n^{(\infty)},H_n^{(\infty)})\}_n$ coincides with the Markov chain $\{(B_n^{(\la)},H_n^{(\la)})\}_n$ up to time $\tau_\la(\eps_0)$. This explains why in the lines below we systematically replace $(B_n^{(\la)},H_n^{(\la)})$ with $(B_n^{(\infty)},H_n^{(\infty)})$. 
Thanks to Lemma \ref{lem:lessommetsretrouves}, we get
\begin{equation*}
(H_n^{(\infty)}-H_{n+1}^{(\infty)})\le \la^{-1}\|V_n^{(\la)}-V_{n+1}^{(\la)}\| \le (H_n^{(\infty)}-H_{n+1}^{(\infty)}) 
+ \la^{-\frac43} \frac{({B_n^{(\infty)}})^2}{H_{n+1}^{(\infty)}}.
\end{equation*}

\noi On $E_{\mbox{\tiny{good}}}^{(\la)}(\eps_0,\eps_1,\eps_2)$, we obtain that
\begin{equation*}
\la^{-\frac43}\frac{({B_n^{(\infty)}})^2}{{H_{n+1}^{(\infty)}}}
\le\la^{-\frac43}\frac{\la^{\eps_1}(B_n^{(\infty)})^2}{(B_{n+1}^{(\infty)})^3}
\le\la^{-\frac43}\frac{\la^{\eps_1+2\eps_2}}{B_{n+1}^{(\infty)}}
\le \la^{-1+\eps_1+2\eps_2-\frac{\eps_0}{3}}
\end{equation*}
which completes the proof of \eqref{eq:normesapprox} with the choice $\delta=\frac{\eps_0'}{3}-\eps_1-2\eps_2>0$ for well-chosen $\eps_0',\eps_1,\eps_2$.

\pass We now prove \eqref{eq:anglesapprox}. Using \eqref{eq:angle2somments} combined with $x\le \arcsin(x)\le x+\frac{x^3}{6}$ for $x\in [0,1]$ and $x-\frac{x^3}{3}\le \arctan(x)\le x$ for $x\ge 0$, we get
\begin{align}\label{eq:majoangle}
\la^{\frac23}\angle (V_n^{(\la)}V_{n+1}^{(\la)}V_{n+2}^{(\la)})-\frac{B_n^{(\infty)}-{B_{n+1}^{(\infty)}}}{{H_{n+1}^{(\infty)}}} 
& \le \la^{-\frac43}\bigg(\frac13\bigg(\frac{B_{n+1}^{(\infty)}}{H_{n+1}^{(\infty)}}\bigg)^3+\frac16\bigg(\frac{B_{n+1}^{(\infty)}}{H_{n+1}^{(\infty)}}\bigg)^3\bigg) \nonumber \\
& \le \frac12 \la^{-\frac43}\bigg(\frac{B_{n+1}^{(\infty)}}{H_{n+1}^{(\infty)}}\bigg)^3 \nonumber \\
& \le \frac12 \la^{-\frac43}\bigg(\frac{{B_n}^{(\infty)}}{H_{n+1}^{(\infty)}}\bigg)^3,
\end{align}
the last inequality coming from the fact that $B_{n+1}^{(\infty)}\leq B_n^{(\infty)}$.

\noi Analogously, we prove the same upper bound for $\frac{B_n^{(\infty)}-B_{n+1}^{(\infty)}}{{H_{n+1}^{(\infty)}}}-\la^{\frac23}\angle (V_n^{(\la)}V_{n+1}^{(\la)}V_{n+2}^{(\la)})$.

\noi Finally, noticing that
\begin{align}\label{eq:petitsauveur}
\frac{B_n^{(\infty)}}{H_{n+1}^{(\infty)}} 
\le \frac1{(B_n^{(\infty)})^2}\frac{(B_n^{(\infty)})^3}{H_{n+1}^{(\infty)}}
\le\la^{\frac23-\frac23\eps_0+3\eps_2}\frac{(B_{n+1}^{(\infty)})^3}{H_{n+1}^{(\infty)}}\le\la^{\frac49-\delta}
\end{align}
where $\delta=\frac23\eps_0'-\eps_1-3\eps_2>0$ for $\eps_1,\eps_2>0$ small enough, we completes the proof of Lemma \ref{lem:lessommetsretrouves2} using \eqref{eq:majoangle}.
\end{proof}

\noi In order to show the convergence of $\{F^{(\la)}(V_n^{(\la)}-Z_{\textsl c}^{(\la)})\}_n$ to $\{V_n^{\mathbf l}\}_n$, we need to truncate the set of points. Consequently, we fix $\eps_0>0$ and consider $\tau_\la(\eps_0)$. In Lemma \ref{lem:aimerjusqualimpossible} below, we show the convergence point by point of the first part of the sequence $\{F^{(\la)}(V_n^{(\la)}-Z_{\textsl c}^{(\la)})\}_n$ up to the stopping time $\tau_\la(\eps_0)$.

\begin{lem}\label{lem:aimerjusqualimpossible}
Let $\eps_0\in(\frac13,1)$. There exists a realization of the sequence $\{V_n^{\mathbf l}\}_n$ such that when $\la\longrightarrow\infty$, we get 
\begin{equation*}
\max_{0 \le n\le \tau_\la(\eps_0)}\|F^{(\la)}(V_n^{(\la)}-Z_{\textsl c}^{(\la)})-V_n^{\mathbf l}\|\overset{\P}{\longrightarrow} 0.
\end{equation*}
\end{lem}

\begin{proof} Remembering that by Proposition \ref{prop:simuquadruplet}, the quadruplet $(R^{(\la)},\Theta_{\textsl l}^{(\la)},\Theta_{\textsl c}^{(\la)},\Theta_{\textsl r}^{(\la)})$ converges in total variation to $(R,\Theta_{\textsl l},\Theta_{\textsl c},\Theta_{\textsl r})$, we introduce the event 
\begin{equation}\label{eq:Einit}
E_{\mbox{\tiny{init}}}^{(\la)}=\big\{(R^{(\la)},\Theta_{\textsl l}^{(\la)},\Theta_{\textsl c}^{(\la)},\Theta_{\textsl r}^{(\la)})=(R,\Theta_{\textsl l},\Theta_{\textsl c},\Theta_{\textsl r})\big\}
\end{equation}
and claim that $\P\big[E_{\mbox{\tiny{init}}}^{(\la)}\big]\underset{\la\to\infty}{\longrightarrow}1$. In the whole proof, we assume that we are on the event $E^{(\la)}_{\mbox{\tiny{good}}}\cap E_{\mbox{\tiny{coupl}}}^{(\la)}\cap E_{\mbox{\tiny{init}}}^{(\la)}$ whose probability converges to $1$. The strategy is the following: we start by using Lemma \ref{lem:lessommetsretrouves2} to obtain a uniform asymptotic estimate on the two coordinates of $F^{(\la)}(V_n^{(\la)})$ in function of the Markov chain $\{(B_n^{(\infty)},H_n^{(\infty)})\}_n$. In a second step, we replace the initial distribution of the Markov chain by its limit distribution, i.e. we rewrite the previous asymptotic estimate in function of $\{(B_n^{\mathbf l},H_n^{\mathbf l})\}_n$. Finally, we find the limits of the coordinates of $F^{(\la)}(Z_{\textsl c}^{(\la)})$ and substract those to the previous asymptotic estimate.

\pass Let us fix $0\le n\le \tau_\la(\eps_0)$. Our first step consists in rewriting the two coordinates of $F^{(\la)}(V_n^{(\la)})$ in function of the sequences $\{\|V_{n+1}^{(\la)}-V_n^{(\la)}\|\}_n$ and $\{\angle (V_n^{(\la)}V_{n+1}^{(\la)}V_{n+2}^{(\la)})\}_n$ as well as the angles $\Theta_{\textsl l}^{(\la)}$ and $\Theta_{\textsl c}^{(\la)}$.

\pass We start with the second coordinate of $F^{(\la)}(V_n^{(\la)})$. Writing
\begin{equation}\label{eq:VnfonctionFn}
V_n^{(\la)}=V_0^{(\la)}+\sum_{k=0}^{n-1}(V_{k+1}^{(\la)}-V_k^{(\la)})
\end{equation}
we obtain
\begin{align}\label{eq:2emecoordV_n}
\la^{-1}\langle e_2,V_n^{(\la)}\rangle
= 1+\la^{-1}\sum_{k=0}^{n-1}\cos\big(\angle (e_2,V_{k+1}^{(\la)}-V_k^{(\la)})\big)\|V_{k+1}^{(\la)}-V_k^{(\la)}\|.
\end{align}

\noi Thanks to \eqref{eq:anglesapprox} and \eqref{eq:petitsauveur}, we obtain uniformly in $k$ between $0$ and $\tau_\la(\eps_0)$ and for some $\delta>0$
\begin{align}\label{eq:decompanglesubtile}
\angle (e_2,V_{k+1}^{(\la)}-V_k^{(\la)}) 
& = \angle (e_2,V^{(\la)}_1-V^{(\la)}_0) + \sum_{j=1}^k \angle (V^{(\la)}_{j-1}V^{(\la)}_jV^{(\la)}_{j+1}) \nonumber \\
& = \mfrac{\la^{-\frac23}}2(\Theta_{\textsl l}^{(\la)}+\Theta_{\textsl c}^{(\la)})+\la^{-\frac23}\sum_{j=1}^k \bigg( \frac{B^{(\la)}_{j-1}-B^{(\la)}_j}{H^{(\la)}_j}+\mathrm{O}(\la^{-\delta}) \bigg) \nonumber \\
& =\mfrac{\la^{-\frac23}}2(\Theta_{\textsl l}+\Theta_{\textsl c})+\la^{-\frac23}\sum_{j=1}^k \bigg( \frac{B^{(\la)}_{j-1}-B^{(\la)}_j}{H^{(\la)}_j}\bigg)+\mathrm{O}(\la^{-\frac23-\delta})\nonumber\\
& = \mfrac{\la^{-\frac23}}2(\Theta_{\textsl l}+\Theta_{\textsl c}) +\mathrm{O}(\la^{-\frac29-\delta}\log\la)\nonumber\\ 
& = \mathrm{O}(\la^{-\frac29-\delta}\log\la).
\end{align}
Consequently, combining \eqref{eq:normesapprox} and \eqref{eq:decompanglesubtile}, we obtain
\begin{align}\label{eq:2emecoordDAinterm}
\la^{-1}\langle e_2,V_n^{(\la)}\rangle
= 1 + \sum_{k=0}^{n-1}(H_{k+1}^{(\infty)}-H_k^{(\infty)})+\mathrm{O}(\la^{-\delta}\log^3\la) = 1 - H_0^{(\infty)} + H_n^{(\infty)} + \mathrm{o}(1)
\end{align}
where $\mathrm{o}(1)$ is uniform with respect to $n\le \tau_\la(\eps_0)$.

\pass We now deal with the first coordinate of $F^{(\la)}(V_n^{(\la)})$. Using \eqref{eq:VnfonctionFn}, \eqref{eq:normesapprox} and \eqref{eq:decompanglesubtile} we get, for some $\delta>0$,
\begin{align}\label{eq:1erecoordDA}
& \la^{-\frac13} \langle e_1, V_n^{(\la)}\rangle 
= \la^{-\frac13}\sum_{k=0}^{n-1}\sin\big(\angle (e_2,V_{k+1}^{(\la)}-V_k^{(\la)})\big)\|V_{k+1}^{(\la)}-V_k^{(\la)}\| \nonumber\\
& = \sum_{k=0}^{n-1}\bigg(\mfrac12(\Theta_{\textsl l}+\Theta_{\textsl c}) + \sum_{j=0}^{k-1}\bigg(\frac{B_j^{(\infty)}-B_{j+1}^{(\infty)}}{H_{j+1}^{(\infty)}}+\mathrm{O}(\la^{-\delta}\log\la)\bigg)\bigg)\bigg(H_k^{(\infty)}-H_{k+1}^{(\infty)}+\mathrm{O}(\la^{-\frac{10}{9}-\delta})\bigg) \nonumber\\
& = \mfrac12(\Theta_{\textsl l}+\Theta_{\textsl c})(H_0^{(\infty)}-H_n^{(\infty)}) + \sum_{j=0}^{n-2}\bigg(\bigg(\sum_{k=j+1}^{n-1}H_k^{(\infty)}-H_{k+1}^{(\infty)}\bigg)\frac{B_j^{(\infty)}-B_{j+1}^{(\infty)}}{H_{j+1}^{(\infty)}} \bigg) +\mathrm{o}(1)\nonumber\\
& = \mfrac12(\Theta_{\textsl l}+\Theta_{\textsl c})(H_0^{(\infty)}-H_n^{(\infty)}) +\sum_{j=0}^{n-2}\big(H_{j+1}^{(\infty)}-H_n^{(\infty)}\big)\frac{B_j^{(\infty)}-B_{j+1}^{(\infty)}}{H_{j+1}^{(\infty)}}+\mathrm{o}(1) \nonumber \\
& = \mfrac12(\Theta_{\textsl l} + \Theta_{\textsl c})(H_0^{(\infty)}-H_n^{(\infty)})+ \big(B_0^{(\infty)}-B_n^{(\infty)}\big) - H_n^{(\infty)}\sum_{j=0}^{n-1}\frac{B_j^{(\infty)}-B_{j+1}^{(\infty)}}{H_{j+1}^{(\infty)}}+\mathrm{o}(1). 
\end{align}

\noi The last step of the proof consists in replacing the Markov chain $\{(B_n^{(\infty)},H_n^{(\infty)})\}_n$ with initial distribution given at \eqref{eq:B0H0} by the Markov chain $\{(B_n^{\mathbf l},H_n^{\mathbf l})\}_n$ defined in part 2 on page \pageref{descBH}. Both Markov chains have same transition probability but different initial distribution and we start by studying the difference between the two initial distributions. To do so, we expand asymptotically the expression of $(B_0^{(\la)},H_0^{(\la)})$ given at \eqref{eq:B0H0} to obtain, when $\la\longrightarrow\infty$,
\begin{align}\label{eq:DAB0H0}
(B_0^{(\la)},H_0^{(\la)}) 
= \big(\mfrac12(\Theta_{\textsl c}-\Theta_{\textsl l}),1\big)+\mathrm{O}(\la^{-\frac43})=(B_0^{\mathbf l},H_0^{\mathbf l})+\mathrm{O}(\la^{-\frac43}).
\end{align}

\noi We now show how to replace in the right-hand side of \eqref{eq:2emecoordDAinterm} and of \eqref{eq:1erecoordDA} all the terms from the sequence $\{(B_n^{(\infty)},H_n^{(\infty)})\}_n$ with the terms from $\{(B_n^{\mathbf l},H_n^{\mathbf l})\}_n$. We concentrate in particular on the final sum in the right-hand side in \eqref{eq:1erecoordDA}. Using \eqref{eq:BnHnrec} and \eqref{eq:BnHnlrec} we obtain, for every $j\ge 0$,
\begin{equation}\label{eq:rapportconstant}
\frac{B_j^{(\infty)}}{B_j^{\mathbf l}}=\frac{\beta_{j-1}\cdots\beta_0 B_0^{(\infty)}}{\beta_{j-1}\cdots\beta_0 B_0^{\mathbf l}}=\frac{B_0^{(\infty)}}{B_0^{\mathbf l}},   
\end{equation}
which implies that
\begin{align}\label{eq:apparitionrapports}
H_n^{(\infty)}\sum_{j=0}^{n-1}\frac{B_j^{(\infty)}-B_{j+1}^{(\infty)}}{H_{j+1}^{(\infty)}}
& = \bigg(\frac{B_0^{(\infty)}}{B_0^{\mathbf l}}\bigg)\bigg(\frac{H_n^{(\infty)}}{H_n^{\mathbf l}}\bigg)H_n^{\mathbf l}\sum_{j=0}^{n-1}\frac{B_j^{\mathbf l}-B_{j+1}^{\mathbf l}}{H_{j+1}^{\mathbf l}}\bigg(\frac{H_{j+1}^{\mathbf l}}{H_{j+1}^{(\infty)}}\bigg).
\end{align}

\noi Let us show that the ratios $\mfrac{B_0^{(\infty)}}{B_0^{\mathbf l}}$ and $\mfrac{H_j^{(\infty)}}{H_j^{\mathbf l}}$, $0\le j\le \tau_\la(\eps_0)$, are close to $1$. First, thanks to \eqref{eq:DAB0H0}, we obtain
\begin{equation}\label{eq:ratioB}
\frac{B_0^{(\infty)}}{B_0^{\mathbf l}}=1+\mathrm{O}(\la^{-\frac43}).    
\end{equation}

\noi Let us show that uniformly for every $0\le j\le \tau_\la(\eps_0)$,
\begin{equation}\label{eq:ratioH}
\frac{H_j^{(\infty)}}{H_j^{\mathbf l}}=1+\mathrm{O}(\la^{-\frac13-\eps_0}).    
\end{equation}

\noi In order to prove \eqref{eq:ratioH}, we use the intermediate sequence $\{T_n^{(\infty)}\}_n$ defined at \eqref{eq:defTninfty} and its corresponding sequence $\{T_n^{\mathbf l}\}_n$ where
\begin{equation}\label{eq:defTndur}
T_n^{\mathbf l}=\frac{(B_n^{\mathbf l})^3}{H_n^{\mathbf l}}.
\end{equation}
Thanks to \eqref{eq:Tnrec} and its analogue for $\{T_n^{\mathbf l}\}_n$ we get, for every $j\ge 0$,
\begin{equation}\label{eq:identityTj}
|T_j^{\mathbf l}-T_j^{(\infty)}|=\beta_{j-1}^3\cdots\beta_0^3|T_0^{\mathbf l}-T_0^{(\infty)}|\le |T_0^{\mathbf l}-T_0^{(\infty)}|.    
\end{equation}

\noi Thanks to \eqref{eq:DAB0H0}, we obtain
\begin{equation}\label{eq:erreurT0}
|T_0^{\mathbf l}-T_0^{(\infty)}| = \mathrm{O}(\la^{-\frac43}).
\end{equation}

\noi We deduce from \eqref{eq:identityTj} and \eqref{eq:erreurT0} that, uniformly for $0\le j\le \tau_\la(\eps_0)$,
\begin{align}\label{eq:erreurTdur}
\bigg|\frac{T_j^{\mathbf l}}{T_j^{(\infty)}}-1\bigg|
\le \frac{H_j^{(\infty)}}{(B_j^{(\infty)})^3}|T_0^{\mathbf l}-T_0^{(\infty)}|
\le \frac{H_0^{(\infty)}}{(\la^{\frac{\eps_0-1}{3}})^3}|T_0^{\mathbf l}-T_0^{(\infty)}|=\mathrm{O}(\la^{-\frac13-\eps_0}).    
\end{align}

\noi Noticing that for every $0\le j\le \tau_\la(\eps_0)$,
\begin{align}\label{eq:HjfnTj}
\frac{H_j^{(\infty)}}{H_j^{\mathbf l}}=\bigg(\frac{B_0^{(\infty)}}{B_0^{\mathbf l}}\bigg)^3\frac{T_j^{\mathbf l}}{T_j^{{(\infty)}}},    
\end{align}
we deduce from \eqref{eq:ratioB} and \eqref{eq:erreurTdur} the result \eqref{eq:ratioH}.

\pass Combining \eqref{eq:apparitionrapports} with \eqref{eq:ratioB} and \eqref{eq:ratioH}, we obtain 
\begin{equation}\label{eq:erreurbigsum}
H_n^{(\infty)}\sum_{j=0}^{n-1}\frac{B_j^{(\infty)}-B_{j+1}^{(\infty)}}{H_{j+1}^{(\infty)}} = H_n^{\mathbf l}\sum_{j=0}^{n-1}\frac{B_j^{\mathbf l}-B_{j+1}^{\mathbf l}}{H_{j+1}^{\mathbf l}}+\mathrm{o}(1).
\end{equation}

\noi The right-hand side of \eqref{eq:2emecoordDAinterm} and the first two terms in the right-hand side of \eqref{eq:1erecoordDA} can be treated analogously. Consequently, we obtain
\begin{align}\label{eq:limcoorde2}
\la^{-1}\langle e_2,V_n^{(\la)}\rangle =  H_n^{\mathbf l} + \mathrm{o}(1)=Y_n^{\mathbf l}+\mathrm{o}(1) \end{align}
and
\begin{align}\label{eq:limcoorde1}
\la^{-\frac13} \langle e_1, V_n^{(\la)}\rangle 
& = \mfrac12(\Theta_{\textsl{l}} 
+ \Theta_{\textsl{c}})(1-H_n^{\mathbf l})
+ (B_0^{\mathbf l}-B_n^{\mathbf l})-H_n^{\mathbf l}\sum_{j=0}^{n-1}\frac{B_j^{\mathbf l}-B_{j+1}^{\mathbf l}}{H_{j+1}^{\mathbf l}} + \mathrm{o}(1) \nonumber \\
& = X_n^{\mathbf l}+\Theta_{\textsl c}+\mathrm{o}(1)
\end{align}
where $X_n^{\mathbf l}$ has been defined at \eqref{eq:defXnl}.

\pass It remains to show that when $\la\longrightarrow\infty$, the coordinates of $F^{(\la)}(Z_{\textsl c}^{(\la)})$ converge to $(\Theta_{\textsl c},0)$. Indeed, we observe that, thanks to Proposition \ref{prop:simuquadruplet},
\begin{align}\label{eq:origineaffinisee}
F^{(\la)}(Z_{\textsl c}^{(\la)})
& = \Big(\la^{-\frac13}(\la+\la^{-\frac13}R^{(\la)})\sin(\la^{-\frac23}\Theta_{\textsl c}^{(\la)}),\la^{-1}(\la-(\la+\la^{-\frac13}R^{(\la)})\cos(\la^{-\frac23}\Theta_{\textsl c}^{(\la)}))\Big) \nonumber \\
& =(\Theta_{\textsl c},0)+\mathrm{O}(\la^{-\frac43}).
\end{align}

\noi Consequently, we deduce from \eqref{eq:limcoorde2}, \eqref{eq:limcoorde1} and \eqref{eq:origineaffinisee} that
\begin{equation*}
F^{(\la)}(V_n^{(\la)}-Z_{\textsl r}^{(\la)})=V_n^{\mathbf l}+\mathrm{o}(1). \end{equation*}

\noi This completes the proof.
\end{proof}

\noi In the next lemma, we show that all the points $V_n^{(\la)}$ after the coupling time $\tau_\la(\eps_0)$ converge to $Z_{\textsl r}^{(\la)}=Z_{\textsl r}$ when $\la\longrightarrow\infty$.

\begin{lem}\label{lem:lesortdespointsaprescouplage}
Let $\eps_0\in(\frac13,1)$. There exists a realization of the sequence $\{V_n^{\mathbf l}\}_n$ such that when $\la\longrightarrow\infty$,
\begin{equation*}
\sup_{n\ge \tau_\la(\eps_0)}\|V_n^{\mathbf l}\|\overset{\P}{\longrightarrow}0
\quad\text{and}\quad\sup_{n\ge \tau_\la(\eps_0)}\|F^{(\la)}(V_n^{(\la)}-Z_{\textsl r}^{(\la)})\|\overset{\P}{\longrightarrow}0.
\end{equation*}
\end{lem}

\begin{proof} 
\noi We start by noticing that both sequences $\{V_n^{\mathbf l}\}_n$ and $\{V_n^{(\la)}\}_n$ generate convex chains. Consequently, it is enough to prove that $\|V_{\tau_\la(\eps_0)}^{\mathbf l}\|\overset{\P}{\longrightarrow}0$ and $\|F^{(\la)}(V_{\tau_\la(\eps_0)}^{(\la)}-Z_{\textsl r}^{(\la)})\|\overset{\P}{\longrightarrow}0$.

\noi Let us first concentrate on $V_{\tau_\la(\eps_0)}^{\mathbf l}$. We start with its second coordinate $Y_{\tau_\la(\eps_0)}^{\mathbf l}$. Thanks to the definitions of $T_n^{\mathbf l}$ and $\tau_\la(\eps_0)$ given at \eqref{eq:defTndur} and \eqref{eq:deftaula} respectively, we obtain
\begin{equation}\label{eq:boundH_ninfty}
Y_{\tau_\la(\eps_0)}^{\mathbf l}=H_{\tau_\la(\eps_0)}^{\mathbf l}=\frac{\big(B_{\tau_\la(\eps_0)}^{\mathbf l}\big)^3}
{T_{\tau_\la(\eps_0)}^{\mathbf l}}.
\end{equation}

\noi Using the definition of $\tau_\la(\eps_0)$, \eqref{eq:rapportconstant} and \eqref{eq:ratioB}, we obtain
\begin{equation}\label{eq:Bnlborneaussi}
B_{\tau_\la(\eps_0)}^{\mathbf l}\le c\la^{-\frac{1+\eps_0}{3}}.    
\end{equation}

\noi Thanks to \eqref{eq:Tnrec}, we get for every $n\ge 0$,
\begin{equation}\label{eq:minorationsimpleT_n}
T_{n}^{\mathbf l}\ge \mfrac32\beta_{n-1}^3\xi_{n-1},
\end{equation}
where we recall that $\{\beta_n\}_n$ and $\{\xi_n\}_n$ are two independent sequences of iid variables which are respectively $\mbox{Beta}(2,2)$ and $\mbox{Exp}(1)$ distributed.

\pass Consequently, we obtain for any $c>0$:
\begin{align}\label{eq:deviationgamma}
\P\big[T_{\tau_\la(\eps_0)}^{\mathbf l}\le (\log\la)^{-3}\big]
& \le  c(\log\la) \P\left[\mfrac32\beta_0^3\xi_0\le (\log\la)^{-3}\right] + \P\big[\tau_\la(\eps_0)> c\log\la\big] \nonumber \\
& \le c(\log\la)^{-\frac12} \E\left[(\mfrac32\beta_0^3\xi_0)^{-\frac12}\right] + \P\big[\tau_{\la}(\eps_0)> c\log\la\big].
\end{align}

\noi Since the second term on the right-hand side of \eqref{eq:deviationgamma} tends to 0 by Lemma \ref{lem:taulainf} for $c$ large enough, we combine \eqref{eq:boundH_ninfty}, \eqref{eq:Bnlborneaussi} and \eqref{eq:deviationgamma} to obtain with a probability going to $1$,
\begin{equation}\label{eq:goodboundH_ninfty}
Y_{\tau_\la(\eps_0)}^{\mathbf l}=H_{\tau_\la(\eps_0)}^{\mathbf l}\le \la^{-1+\eps_0}\log^3\la,
\end{equation}
This shows that $Y_{\tau_\la(\eps_0)}^{\mathbf l}\overset{\P}{\longrightarrow}0$.

\noi We now treat the first coordinate $X_{\tau_\la(\eps_0)}^{\mathbf l}$ of $V_{\tau_\la(\eps_0)}^{\mathbf l}$. Using successively \eqref{eq:defXnl}, \eqref{eq:relrecpratiquepourplustard} and the fact that $\{B_n^{\mathbf l}\}_n$ is decreasing, we obtain
\begin{align}\label{eq:Abel2}
|X_{\tau_\la(\eps_0)}^{\mathbf l}| & =\bigg|\mfrac12(\Theta_{\textsl l}+\Theta_{\textsl c})H_{\tau_\la(\eps_0)}^{\mathbf l}
+ B_{\tau_\la(\eps_0)}^{\mathbf l} + H_{\tau_\la(\eps_0)}^{\mathbf l}\sum_{k=0}^{{\tau_\la(\eps_0)}-1}\frac{B_k^{\mathbf l}-B_{k+1}^{\mathbf l}}{H_{k+1}^{\mathbf l}}\bigg| \nonumber \\
& =\bigg|\bigg(\mfrac12(\Theta_{\textsl l}+\Theta_{\textsl c})+\frac{B_0^{\mathbf l}}{H_1^{\mathbf l}}\bigg)H_{\tau_\la(\eps_0)}^{\mathbf l}+H_{\tau_\la(\eps_0)}^{\mathbf l}\sum_{k=1}^{{\tau_\la(\eps_0)}-1}B_k^{\mathbf l}\bigg(\frac1{H_{k+1}^{\mathbf l}}-\frac1{H_k^{\mathbf l}}\bigg)
 \bigg| \nonumber \\
& = \bigg|\bigg(\mfrac12(\Theta_{\textsl l}+\Theta_{\textsl c})+\frac{B_0^{\mathbf l}}{H_1^{\mathbf l}}\bigg)H_{\tau_\la(\eps_0)}^{\mathbf l}+ \frac32H_{\tau_\la(\eps_0)}^{\mathbf l}\sum_{k=1}^{{\tau_\la(\eps_0)}-1}\frac{\xi_k}{(B_k^{\mathbf l})^2}\bigg|\nonumber\\
&\le \bigg|\mfrac12(\Theta_{\textsl l}+\Theta_{\textsl c}) + \frac{B_0^{\mathbf l}}{H_1^{\mathbf l}}\bigg| H_{\tau_\la(\eps_0)}^{\mathbf l}
+ \mfrac32 \frac{B_{\tau_\la(\eps_0)}^{\mathbf l}}{T_{\tau_\la(\eps_0)}^{\mathbf l}} {\tau_\la(\eps_0)} \sup_{k\le {\tau_\la(\eps_0)}}\xi_k.
\end{align}

\noi We treat separately each term in the right-hand side of \eqref{eq:Abel2}. Thanks to \eqref{eq:goodboundH_ninfty}, we get
\begin{equation}\label{eq:1erstermescvprobanla}
\bigg|\mfrac12(\Theta_{\textsl l}+\Theta_{\textsl c})
+ \frac{B_0^{\mathbf l}}{H_1^{\mathbf l}}\bigg|H_{\tau_\la(\eps_0)}^{\mathbf l}
\overset{\P}{\longrightarrow}0.    
\end{equation}

\noi We recall that when $n\to \infty$, the sequence $\big\{\sup_{1\le k\le n}\xi_k -\log n\big\}_n$ converges in distribution to a Gumbel variable. Consequently, thanks to Lemma \ref{lem:taulainf} and the monotonicity of the sequence $\big\{\sup_{1\le k\le n}\xi_k\big\}_n$, we obtain that
\begin{equation}\label{eq:interm3emetermerhscasnla}
\lim_{\la\to\infty}\P\bigg[\tau_\la(\eps_0) \sup_{1\le k\le \tau_\la(\eps_0)}\xi_k\le 2\big(\mfrac25 \log\la\big)\log\big(\mfrac25\log\la\big)\bigg]=1.
\end{equation}

\noi Using \eqref{eq:Bnlborneaussi}, \eqref{eq:deviationgamma} and \eqref{eq:interm3emetermerhscasnla}, we get that for any $\eps_0'\in (\eps_0-\frac13,\frac23)$,
\begin{equation}\label{eq:3emetermenlacvP0}
\lim_{\la\to\infty}\P\bigg[\frac{B_{\tau_\la(\eps_0)}^{\mathbf l}}{T_{\tau_\la(\eps_0)}^{\mathbf l}}\tau_\la(\eps_0)\sup_{1\le k\le \tau_\la(\eps_0)}\xi_k\le\la^{-\frac29+\frac{\eps_0'}{3}}\bigg]=1.
\end{equation}
It remains to combine \eqref{eq:Abel2}, \eqref{eq:1erstermescvprobanla} and \eqref{eq:3emetermenlacvP0} to deduce that $X_{\tau_\la(\eps_0)}\overset{\P}{\longrightarrow}0$. This completes the proof of the first part of Lemma \ref{lem:lesortdespointsaprescouplage}.

\pass We now study $\|F^{(\la)}(V_{\tau_\la(\eps_0)}^{(\la)}-Z_{\textsl r}^{(\la)})\|$. Because of Lemma \ref{lem:aimerjusqualimpossible}, we get 
\begin{equation}\label{eq:fado}
\|F^{(\la)}(V_{\tau_\la(\eps_0)}^{(\la)}-Z_{\textsl r}^{(\la)})-V_{\tau_\la(\eps_0)}^{\mathbf l}\|\overset{\P}{\longrightarrow}0.
\end{equation}

\noi Consequently, since $\|V_{\tau_\la(\eps_0)}\|\overset{\P}{\longrightarrow}0$, we obtain 
\begin{equation}\label{eq:departratatinementVn}
\|F^{(\la)}(V_{\tau_\la(\eps_0)}-Z_{\textsl r}^{(\la)})\|\overset{\P}{\longrightarrow}0.
\end{equation}
This completes the proof. 
\end{proof}

\subsection{Proof of Theorem \ref{theo:menhir} (ii): local behavior at the origin}\label{sec:preuveTH2}

\pass 

\noi Let us first notice that 
\begin{equation}\label{eq:YnXnavecWn}
\frac{Y_n^{\mathbf l}}{(X_n^{\mathbf l})^3}=\frac{(T_n^{\mathbf l})^2}{(W_n^{\mathbf l})^3}    
\end{equation}
where the sequence $\{W_n^{\mathbf l}\}_n$ is defined by
\begin{equation*}
W_n^{\mathbf l}=\frac{(B_n^{\mathbf l})^2}{H_n^{\mathbf l}}X_n^{\mathbf l}.  
\end{equation*}

\noi We show that $\{W_n^{\mathbf l}\}_n$ satisfies the following recursion identity, analogue to \eqref{eq:Tnrec}:
\begin{equation}\label{eq:Wnrec}
W_{n+1}^{\mathbf l}=\beta_n^2(W_n^{\mathbf l}+\tfrac32\xi_n).    
\end{equation}

\noi Indeed, 
the recursion relation satisfied by $\{X_n^{\mathbf l}\}_n$ and stated in part 3 on page \pageref{old2}
and \eqref{eq:relrecpratiquepourplustard} imply that 
\begin{align*}
\frac{X_{n+1}^{\mathbf l}}{H_{n+1}^{\mathbf l}}-\frac{X_n^{\mathbf l}}{H_n^{\mathbf l}}
= B_n^{\mathbf l}\Big(\frac1{H_{n+1}^{\mathbf l}}-\frac1{H_n^{\mathbf l}}\Big) = \frac{\frac32\xi_n}{(B_n^{\mathbf l})^2}
\end{align*}
which is equivalent to \eqref{eq:Wnrec}. Proposition \ref{prop:unbusrapide} extends Proposition \ref{prop:loiTn} to the convergence in distribution of the couple $(W_n^{\mathbf l},T_n^{\mathbf l})$.

\begin{prop}\label{prop:unbusrapide}
When $n\to\infty$, the sequence $\{(W_n^{\mathbf l},T_n^{\mathbf l})\}_n$ converges in distribution to a non-degenerate distribution which does not depend on the value of $(W_0^{\mathbf l},T_0^{\mathbf l})$.
\end{prop}

\begin{proof}
As for the convergence of $T_n^{\mathbf l}$, we use Letac's principle and show that for every $n\ge 0$, $(W_n^{\mathbf l},T_n^{\mathbf l})$ has same distribution as $(W_n,T_n)$ where $(W_0,T_0)=(W_0^{\mathbf l},T_0^{\mathbf l})$ and
\begin{equation*}
(W_n,T_n)=(\beta_1^2(\cdots (\beta_{n-1}^2(W_0+\tfrac32\xi_{n-1})+\cdots)+\tfrac32\xi_1),\beta_1^3(\cdots (\beta_{n-1}^3(T_0+\tfrac32\xi_{n-1})+\cdots)+\tfrac32\xi_1)).
\end{equation*}
Since $\{(W_n,T_n)\}_n$ converges almost surely to $\frac32\sum_{n=1}^{\infty}\xi_n\beta_1^2\cdots\beta_n^2(1,\beta_1\cdots\beta_n)$, we deduce that $\{(W_n^{\mathbf l},T_n^{\mathbf l})\}_n$ converges in distribution to the distribution of this random series.
\end{proof}

\pass Finally, combining \eqref{eq:YnXnavecWn} and Proposition \ref{prop:unbusrapide}, we complete the proof of Theorem \ref{theo:menhir} (ii).

\subsection{Proof of Theorem \ref{theo:menhir} (iii): total number of vertices}\label{sec:preuveTH3}

\pass 

\noi As in the proof of (i), Theorem \ref{theo:menhir} (iii) is a direct consequence of Proposition \ref{prop:cestpossible} below on the number of edges $\cN(\cB^{(\la)})$ of the left branch $\cB^{(\la)}$. 

\begin{prop}\label{prop:cestpossible}
When $\la\longrightarrow\infty$,
\begin{equation}\label{eq:limitlengthbranch}
\frac{\cN(\cB^{(\la)})}{\frac25\log\la}\overset{\P,L^1}{\longrightarrow} 1.
\end{equation}
\end{prop}

\begin{proof} 
The idea is to compare $\cN(\cB^{(\la)})$ to $\tau_\la(\eps_0)$ for $\eps_0$ small enough and use Lemma \ref{lem:taulainf}. Namely, we wish to show that there exists a constant $C>0$ such that, for all $\eps_0>0$,
\begin{equation}\label{eq:cestpossibleproof1}
\lim_{\la\to\infty}\P\Big[\big\{\cN(\cB^{(\la)}) -\tau_\la(\eps_0)> C\eps_0\log\la\big\}\cap E_{\mbox{\tiny{coupl}}}^{(\la)}\Big]=0.
\end{equation}

\noi Indeed, for all $n\ge \tau_\la(\eps_0)$, conditional on $\{\cN(\cB^{(\la)})>n\}=\{V_n^{(\la)} \in \bH\}$, $Z_n^{(\la)}$ is on the right (resp. on the left) of the bisecting line of $(Z_{n-1}^{(\la)},Z_{\textsl c}^{(\la)})$ with probability $\frac12$. When $V_n^{(\la)} \in \bH$, the circular arc centered at $V_n^{(\la)}$ and joining $Z_{n-1}^{(\la)}$, $Z_n^{(\la)}$ and $Z_{\textsf c}^{(\la)}$ is smaller than a half-circle. Consequently,
\begin{equation}\label{eq:probaconfin}
\P\Big[B_n^{(\la)}\le \mfrac1{\sqrt2} B_{n-1}^{(\la)} \,\Big|\, \cN(\cB^{(\la)})>n \Big] > \frac12.
\end{equation}

\noi Let $R_*$ be the distance from $Z_{\textsl c}^{(\la)}$ to its nearest neighbor in the initial Poisson point process $\cP$. On the event $E_{\mbox{\tiny{coupl}}}^{(\la)}$ it holds that $B^{(\la)}_{\tau_\la(\eps_0)}=B^{(\infty)}_{\tau_\la(\eps_0)}< \la^{\frac{\eps_0 -1}{3}}$. Hence, after
\begin{equation*}
D=2\log_2 \big (2 \mfrac{\la^{\frac{\eps_0}{3}}}{R_*}\big)
\end{equation*}
divisions by $\sqrt2$, we would get a distance from $Z_{\textsl c}^{(\la)}$ at most equal to $\mfrac{R_*}{2}$ which happens with probability $0$.

\pass By \eqref{eq:probaconfin}, on the event $\big\{\cN(\cB^{(\la)})-\tau_\la(\eps_0)>C\eps_0 \log\la\big\}\cap E_{\mbox{\tiny{coupl}}}^{(\la)}$, the number of divisions by $\sqrt2$ of the distance from $Z_{\textsl c}^{(\la)}$ between times $\tau_\la(\eps_0)$ and $\cN(\cB^{(\la)})$ stochastically dominates a random variable $N$ following a Binomial distribution $\mbox{Bin}(C\eps_0\log\la,\frac12)$. Moreover, because of the previous remark, this number is upper bounded by $D$. Consequently, we get
\begin{equation*}
\P\Big[\big\{\cN(\cB^{(\la)}) -\tau_\la(\eps_0)> C\eps_0\log\la\big\}\cap E_{\mbox{\tiny{coupl}}}^{(\la)}\Big]\le \P[N\le D].
\end{equation*}
It remains to notice that this probability goes to $0$ as soon as $C\eps_ 0\log\la > 3D$, which happens for $C=\frac{8}{3\log 2}$ and when $\la$ is large enough. This proves \eqref{eq:cestpossibleproof1}.

\pass We now fix $\eta>0$, take $C=\frac{8}{3\log 2}$ as above and choose $\eps_0=\frac{\eta}{2C}$. We get
\begin{equation*}
\P\Big[|\cN(\cB^{(\la)})-\mfrac25\log\la|>\eta \log\la\Big ]\le T_1(\la)+T_2(\la)+T_3(\la)
\end{equation*}
where
\begin{align*}
T_1(\la) & =\P\Big[\big\{|\cN(\cB^{(\la)})-\mfrac25\log\la|>\eta\log\la\,;\,\cN(\cB^{(\la)})-\tau_\la(\eps_0) < C\eps_0 \log\la\big\}\cap E_{\mbox{\tiny{coupl}}}^{(\la)} \Big], \\
T_2(\la) & =\P\Big[\big\{\cN(\cB^{(\la)})-\tau_\la(\eps_0)>C\eps_0 \log\la \big\}\cap E_{\mbox{\tiny{coupl}}}^{(\la)}\Big] \\
\text{and}\hspace{1cm} T_3(\la) & =\P\big[(E_{\mbox{\tiny{coupl}}}^{(\la)})^c\big].\hfill
\end{align*}

\noi By \eqref{eq:cestpossibleproof1} we get $\lim_{\la\to\infty} T_2(\la)=0$ and by Proposition \ref{prop:evenpascouplage}, we obtain $\lim_{\la\to\infty}T_3(\la)=0$.

\noi Finally, we notice that
\begin{equation*}
T_1(\la)\le \P\Big[\big|\tau_\la(\eps_0)-\mfrac25\log\la\big|>\mfrac{\eta}{2}\log\la\Big]
\end{equation*}
and this goes to zero by Lemma \ref{lem:taulainf}. In conclusion, we get the convergence in probability in \eqref{eq:limitlengthbranch}. 

\pass Thanks to \eqref{eq:probaconfin}, the variable $\cN(\cB^{(\la)})$ is stochastically dominated by a sum of $C'(\log\la+\log B_0^{(\la)})$ iid geometric variables with probability parameter $\frac12$. The variable $\log B_0^{(\la)}$ is then upper bounded up to an additive constant by $\log R^{(\la)}$. Since the variable $R^{(\la)}$ satisfies
\begin{equation*}
\P\big[R^{(\la)}\ge r\big] \leq C''\exp\big(-\mfrac{r^{\frac32}}{C''}\big)
\end{equation*}
for $r>0$ large enough, the sequence $\big\{\log R^{(\la)}\big\}_\la$ is bounded in $L^2$. Consequently, $\big\{\frac{\cN(\cB^{(\la)})}{\log\la}\big\}_\lambda$ is bounded in $L^2$ too. The $L^1$-convergence is then a consequence of the convergence in distribution and uniform integrability of $\big\{\frac{\cN(\cB^{(\la)})}{\log \la}\big\}_\la$.
\end{proof}

\bibliographystyle{plain}

\end{document}